\documentclass[imslayout]{imsart}

\usepackage{amsmath,amsfonts,amsthm,mathrsfs,amssymb}
\usepackage{natbib}
\usepackage[utf8]{inputenc}
\usepackage[T1]{fontenc} 
\usepackage{fullpage}
\usepackage{graphicx}
\newtheorem{theorem}{Theorem}[section]
\newtheorem{proposition}[theorem]{Proposition}
\newtheorem{corollary}[theorem]{Corollary}
\newtheorem{definition}[theorem]{Definition}
\newtheorem{lemma}[theorem]{Lemma}
\newtheorem{remark}[theorem]{Remark}

\DeclareMathOperator*{\argmax}{arg\,max}

\begin{document}

\begin{frontmatter}
\title{Asymptotic properties of the maximum likelihood estimator for multivariate extreme value distributions}

\begin{aug}
  \author{\fnms{Cl\'ement} \snm{Dombry}\thanksref{m1}\ead[label=e1]{clement.dombry@univ-fcomte.fr}}
  \and
  \author{\fnms{Sebastian} \snm{Engelke}\thanksref{m2}\ead[label=e2]{sebastian.engelke@epfl.ch}}
  \and
  \author{\fnms{Marco} \snm{Oesting}\thanksref{m3}\ead[label=e3]{oesting@mathematik.uni-siegen.de}}
  
  \affiliation{Universit\'e Bourgogne Franche--Comt\'e\thanksmark{m1}, Ecole Polytechnique F\'ed\'erale de Lausanne\thanksmark{m2} and Universit\"at Siegen\thanksmark{m3}}

  \address{Universit\'e Bourgogne Franche--Comt\'e\\
    Laboratoire de Math\'ematiques de Besan\c{c}on\\
    UMR CNRS 6623\\
    16 Route de Gray\\
    25030 Besan\c{c}on cedex\\
    France\\
    \printead{e1}}  
  \address{Ecole Polytechnique F\'ed\'erale de Lausanne\\
    EPFL-FSB-MATHAA-STAT\\
    Station 8\\
    1015 Lausanne\\
    Switzerland\\
    \printead{e2}}
  \address{Universit\"at Siegen\\
    Department of Mathematics\\
    Walter-Flex-Str.~3\\
    57068 Siegen\\
    Germany\\
    \printead{e3}}

  \runauthor{C.~Dombry, S.~Engelke and M.~Oesting}
\end{aug}

\begin{abstract}
  Max-stable distributions and processes are important models for extreme events and the assessment of tail risks. 
  The full, multivariate likelihood of a parametric max-stable distribution is complicated and only recent advances 
  enable its use. The asymptotic properties of the maximum likelihood estimator in multivariate extremes are mostly unknown.
  In this paper we provide natural conditions on the exponent function and the angular measure of the max-stable distribution
  that ensure asymptotic normality of the estimator. We show the effectiveness of this result by applying it to popular 
  parametric models in multivariate extreme value statistics and to the most commonly used families of spatial max-stable processes.
\end{abstract}

\begin{keyword}[class=MSC]
\kwd[Primary: ]{62F12 }
\kwd[Secondary: ]{60G70}
\end{keyword}

\begin{keyword}
\kwd{asymptotic normality}
\kwd{differentiability in quadratic mean}
\kwd{max-stable distribution}
\end{keyword}

\end{frontmatter}

\section{Introduction} \label{sec:intro}

The theory of multivariate and spatial extremes has been rapidly evolving
in the last decades. The two main approaches are to consider threshold exceedances
resulting in multivariate Pareto distributions \citep{roo2006}, or to approximate 
componentwise block maxima by max-stable distributions.
A $k$-dimensional random vector $Z = (Z_1,\dots,Z_k)$ with unit Fr\'echet margins is
called max-stable if it arises as the normalized limit of the componentwise maxima of
some sequence of independent, identically distributed random vectors. This property 
makes max-stable distributions a natural model in statistics to describe the joint 
behavior of multivariate extreme events. Extensions to the continuous domain, the 
so-called max-stable processes, are widely applied in meteorology and hydrology to
assess the risk of temporal or spatial dependence between extreme observations 
\citep[see, e.g.,][]{bui2008, dav2012b, eng2014, asa2015, oes2016}.

The distribution function $F$ of a simple max-stable random vector $Z$
can be written in the form $F(z) =\exp\{-V(z)\}$, $z\in (0,\infty)^k$, where $V$ is
the so called exponent function. It describes the dependence between the components
of $Z$ and satisfies a homogeneity property \citep{res2008}. If $F$ admits a 
continuous density $f$ then it is obtained by taking partial derivatives of $F$ 
with respect to all components. By Fa\`a di Bruno formula, this yields
\begin{equation} \label{eq:llh_Z}
  f(z) = \sum_{\tau\in \mathscr{P}_k} \exp\{-V(z)\} \prod_{j=1}^{|\tau|} \{-\partial_{\tau_j}V(z)\},
\end{equation}
where $\mathscr{P}_k$ is the set of all partitions $\tau = \{\tau_1,\ldots,
\tau_{|\tau|}\}$ of $\{1,\ldots,k\}$ and $\partial_{\tau_j}V(\cdot)$ 
denotes the partial derivative of the exponent function $V$ with respect to the
variables $z_i$, $i\in \tau_j$.

In parametric extreme value statistics the distribution $F$ is modeled
by a parametric family $\{F_\theta,\ \theta \in \Theta\}$ of max-stable distribution 
functions with corresponding densities $f_\theta(z)$.
When considering the maximum likelihood estimator 
\begin{align*}
  \hat\theta_n^{\rm mle} = \argmax_{\theta\in\Theta} \prod_{i=1}^n f_\theta(Z^{(i)})
\end{align*}
for independent observations $Z^{(1)},\dots, Z^{(n)}$ of $Z$, the 
combinatorial explosion of $\mathscr{P}_k$ renders the computation of the 
likelihood and its maximization challenging. Indeed, the 
number of terms in \eqref{eq:llh_Z} equals the $k$th Bell number which grows
super-exponentially in the dimension~$k$. A common way to avoid this problem is
to consider composite pairwise likelihoods instead of full likelihoods. 
Computation of this composite likelihood relies only on the densities of 
bivariate sub-vectors of $Z$ \citep{PRS10}. Apart from the fact that it is 
misspecified, this approach can lead to considerable losses in efficiency \citep{CHG16}.
If additional information on the partition $\tau$ is available, then a
simplified likelihood consisting of only one summand in \eqref{eq:llh_Z} can be
used \citep{ST05, wadsworth14}.

A recent approach allowing the use of full likelihoods is to introduce
a prior distribution $\pi_{\mathrm{prior}}(\mathrm{d}\theta)$
on $\Theta$ and to treat the partition $\tau$ as a latent variable in a 
hierarchical Bayesian framework \citep{thi2015,deo2016}. 
Given observations $Z^{(1)},\ldots,Z^{(n)}$, the posterior distribution 
\[
\pi_{\mathrm{post}}(\mathrm{d}\theta\mid Z^{(1)},\ldots,Z^{(n)})\propto \pi_{\mathrm{prior}}(\mathrm{d}\theta)\, \prod_{i=1}^n f_\theta(Z^{(i)}) 
\]
can be assessed by Markov chain Monte Carlo simulations and point estimators for 
$\theta$ can be obtained as a functional of the posterior distribution such as the
mean or median. This method results in a significant gain of  efficiency compared to 
composite pairwise likelihood methods \citep{deo2016}. The methodology for computing 
the Bayes likelihood estimator relies on conditional simulation \citep{DEMR13}. 
Besides the Bayesian framework, inference based on full likelihood has recently 
become available also in a frequentist setting via a stochastic Expectation-Maximization
algorithm proposed by \citet{dombry-etal-17}.

Even in the univariate case, the asymptotic theory of likelihood estimators in 
extreme value statistics is non-standard. The fact that the support of the 
generalized extreme value distribution depends on the parameter values makes 
the theory difficult \citep{smi1985,BS16a,DF17}. In the multivariate case, the 
asymptotic properties of likelihood estimators for the parameter $\theta$ of
the max-stable parametric family $\{F_\theta,\ \theta \in \Theta\}$ are mostly
unknown. Asymptotic here always means that the number $n$ of independent 
observations of the max-stable vector $Z$ tends to infinity. The composite 
maximum likelihood method in \cite{PRS10} is asymptotically normal under 
certain regularity conditions, but verifying these conditions seems difficult 
and has not been done for any of the existing models. For the full maximum 
likelihood estimator $\hat\theta_n^{\rm mle}$ in dimensions $k\geq2$, no theory 
exists to the best of our knowledge, with the exception of \cite{taw1988a} who
treats the two-dimensional logistic model. \cite{bie2016} consider asymptotic 
normality under technical assumptions that, again, are hard to verify.
The main reason for the lack of results in the literature is the complicated 
form of the likelihood \eqref{eq:llh_Z} that makes the theoretical analysis 
difficult even for small dimensions $k$. Moreover, until recently, the use of
full likelihoods was computationally infeasible because of the explosion of the
number of summands.  

In this paper we provide general conditions for the consistency and asymptotic normality 
of the maximum likelihood estimator $\hat\theta_n^{\rm mle}$ for multivariate max-stable
distributions. As a preliminary step, we obtain in Section \ref{sec:metho} necessary and 
sufficient conditions for the existence of the density of 
a simple max-stable distribution $Z$. The conditions are natural in the sense that they 
are formulated in terms of the exponent function and the angular density of the max-stable 
model, two objects that are typically used to characterize the distribution of~$Z$.

Modern likelihood estimation theory relies on an important regularity property 
called differentiability in quadratic mean. In Section \ref{sec:quad} we provide
simple conditions for this property, again stated in terms of the exponent 
function and the angular density. Most importantly, these conditions can be 
verified for all popular parametric models and are sufficiently general to apply
to other models as well. The consistency and asymptotic normality for the maximum
likelihood estimator $\hat\theta_n^{\rm mle}$ are then established in Section
\ref{sec:MLE}.

In Section \ref{sec:exam} we apply our general conditions to show asymptotic
normality for the most popular models in multivariate extreme value theory, 
such as the $k$-dimensional logistic, Dirichlet and H\"usler--Reiss
distributions. In the spatial domain we cover common parameterizations of the 
Schlather and Brown--Resnick max-stable processes.

\section{Densities of max-stable distributions} \label{sec:metho}

\subsection{Max-stable distributions}\label{sec:max}
A $k$-dimensional random vector $Z= (Z_1,\dots,Z_k)$ with distribution function $F$
is called max-stable if it arises as the
limit in distribution of the normalized, componentwise maxima of
some sequence of independent, identically distributed random vectors.
The distribution of $Z$ can be assumed to be simple, that is, to have standard Fr\'echet
margins $\mathbb P(Z_i \leq z) = \exp( - 1/z)$ \citep[Chap.~5]{res2008}.
The max-stability then means that for $n$ independent copies $Z^{(1)},\dots, Z^{(n)}$
of $Z$, the vector of componentwise maxima $\max_{j=1,\dots, n} Z^{(j)}$ has the
same distribution as~$nZ$. 

The law of a simple max-stable random vector $Z$ can be characterized
by any of the following objects.
\begin{itemize}
\item The exponent function $V:E^k \to (0,\infty)$ of $Z$, where $E^k = [0,\infty)^k\setminus\{0\}$,
  is defined by
  \[
  V(z) = -\log F(z), \quad z\in E^k.
  \]
  It is homogeneous of order $-1$, i.e., $V(uz)=u^{-1}V(z)$ for all $u>0$, and 
  satisfies the normalization condition $V(\infty,\ldots, 1,\ldots,\infty)=1$.
\item
  The exponent measure $\Lambda$ on $E^k$ is related to $V$ by
  \begin{equation}\label{eq:V}
  \Lambda\left( E^k\setminus [0,z]\right)=V(z), \quad z\in E^k.
  \end{equation}
  The homogeneity property reads $\Lambda(uA)=u^{-1}\Lambda(A)$ for all Borel set $A\subset E^k$ and $u>0$. 
	The exponent measure is normalized by $\Lambda(\{z\in E^k: \ z_i>1\})=1$, for any $i=1,\dots,k$.
\item 
  The angular measure $H$ of $Z$ is a probability measure on the simplex
  $$S^{k-1}=\{w\in[0,1]^k: \ w_1+\cdots+w_k=1\}$$ and related to the exponent measure by
  \begin{equation}\label{eq:H}
  \Lambda( \{z\in E^k: \ \|z\|_1 > r, \ z/\|z\|_1 \in B \})= k r^{-1}H(B)
  \end{equation}
  for any Borel set $B\subset S^{k-1}$ and $r>0$. It satisfies the moment constraint
  \[
  \int_{S^{k-1}}w_iH(dw)=\frac{1}{k}.
  \]
\end{itemize}
The normalization constraints in the above definitions all ensure that the 
marginal distributions of $Z$ are standard Fr\'echet. The exponent measure
$\Lambda$ is the intensity measure of a Poisson point process $\Pi = \{\psi _i: i\in \mathbb N\}$
on the space $E^k$. It is generating the max-stable distribution $Z$ in the sense
that $Z = \max_{i\in\mathbb N} \psi_i$ is the componentwise maximum of all points in $\Pi$.
Many developments in multivariate extreme value statistics are based on a
detailed analysis of the point process $\Pi$, including exact and conditional
simulation of max-stable models \citep{DEM13,dom2016}.

There exist many parametric classes of max-stable distributions, such as the 
logistic \citep{gum1960}, the H\"usler--Reiss \citep{hue1989} and the Dirichlet
models \citep{col1991, bol2007}. More recently, max-stable distributions have been
extended to the temporal and spatial domain and popular parametric models include
the Schlather \citep{sch2002}, the Brown--Resnick \citep{bro1977,kab2009} and the 
extremal-$t$ processes \citep{opi2013}. All these models will be defined in Section~\ref{sec:exam}
where we apply our general results to specific examples.

\subsection{Existence of densities}

In this section, we investigate the existence of the density of the distribution
of a $k$-dimensional max-stable random vector, i.e., the existence of a measurable
function $f: (0,\infty)^k \to [0,\infty)$ such that
$$ F(z) = \int_{(0,z)} f(y) \, \mathrm{d}y, \quad z \in (0,\infty)^k, $$
where $F$ is the cumulative distribution function of $Z$ and, thus, of the form 
$F(z)=\exp\{-V(z)\}$. If we assume that the multivariate derivative
\[
f(z)=\frac{\partial^k F}{\partial z_1\cdots\partial z_k}(z), \quad z \in (0,\infty)^k,
\]  
exists and is continuous, then it equals the desired density and it is given by the 
right hand side of Equation~\eqref{eq:llh_Z}. This is, however, only a sufficient
condition for the existence of a density. In the following, we provide necessary 
and sufficient conditions for existence of the density of $Z$ in terms the existence
of densities of the exponent measure $\Lambda$ and the angular measure $H$. Since
these objects are typically used to characterize families of max-stable 
distributions (cf., Section \ref{sec:max}), the conditions can readily be verified
for all existing models.

We first introduce some notation. For $z\in E^k$ and $I\subset\{1,\ldots,k\}$, we denote 
by $z_I$ and $z_{I^c}$ the sub-vectors of $z$ with components in $I$ and 
$I^c=\{1,\ldots,k\}\setminus I$ respectively. The set $E^k$ is the disjoint union 
of the faces
\[
E^k_I=\{z\in E^k: \ z_i>0 \mbox{ for } i \in I \mbox{ and } z_{i}=0 \mbox{ for } i \notin I\}, \quad \emptyset\neq I\subset \{1,\ldots,k\}.
\]
For $I$ with cardinality $|I| = d\in\{1,\ldots,k\}$, $E^k_I$ is a $d$-dimensional face of
$E^k$ and we denote by $\mu_I(\mathrm{d}z)=\mathrm{d}z_I\delta_0(\mathrm{d}z_{I^c})$,
the Lebesgue measure on  $E^k_I$. Similarly, the simplex $S^{k-1}$ can be decomposed
into faces
\[
S^{k-1}_I=\{w\in S^{k-1}: \ w_i>0 \mbox{ for } i \in I \mbox{ and } w_i=0 \mbox{ for } i\notin I\}, \quad \emptyset\neq I\subset \{1,\ldots,k\}.
\]
For $|I|=d \in\{1,\ldots,k\}$, $S^{k-1}_I$ is a $(d-1)$-dimensional face of 
$S^{k-1}$ and we denote by $\sigma_I(\mathrm{d}w)=\mathrm{d}w_I \delta_0(\mathrm{d}w_{I^c})$ 
the Lebesgue measure on $S^{k-1}_I$.
In the case $d=1$, the face $S^{k-1}_I$ is a  vertex of the simplex and 
$\sigma_I(\mathrm{d}w)$ denotes the Dirac mass at this vertex; see \cite{col1991} for details
on angular densities.

\begin{proposition}\label{prop1}
The following statements are equivalent:
\begin{itemize}
\item[i)] the multivariate simple max-stable distribution $F$ admits a density $f$;
\item[ii)] the exponent measure $\Lambda$ admits a density $\lambda_I$ on each face $E^k_I$, i.e.,
\[
\Lambda(A)=\sum_{\emptyset\neq I\subset\{1,\ldots,k\}} \int_{A} \lambda_I(z)\mu_I(\mathrm{d}z),\quad\mbox{for all Borel set } A\subset E^k;
\]
\item[iii)] the angular measure $H$ admits a density $h_I$ on each face $S^{k-1}_I$, i.e.,
\[
H(B)=\sum_{\emptyset \neq I\subset\{1,\ldots,k\}} \int_{B} h_I(w)\sigma_I(\mathrm{d}w),\quad\mbox{for all Borel set } B\subset S^{k-1}.
\]
\end{itemize}
In this case, the density $f(z)$, $z\in (0,\infty)^k$, is given by the right hand side of \eqref{eq:llh_Z} with
\begin{eqnarray}
V(z)&=&\sum_{\emptyset\neq I\subset\{1,\ldots,k\}} \int_{E^k\setminus [0,z] } \lambda_I(u)\mu_I(\mathrm{d}u),\label{eq:V2}\\
-\partial_{\tau_j}V(z)&=& \sum_{\tau_j\subset I\subset\{1,\ldots,k\}}  \int_{[0,z_{\tau_j^c})} \lambda_I(z_{\tau_j},u_j)\mathrm{d}u_{j,I}\delta_0(\mathrm{d}u_{j,I^c}) .\label{eq:Vtau}
\end{eqnarray}
Furthermore, the densities $\lambda_I$ and $h_I$ are related by the following homogeneity relations: 
\begin{equation}
\lambda_I(z)=k \|z\|_1^{-|I|-1}h_I(z/\|z\|_1),\quad z\in E_I^k. \label{eq:lambda}
\end{equation}
\end{proposition}

\begin{proof}
Using a detailed analysis of the Poisson point process representation of max-stable 
processes, \cite{DEM13} derived a general formula for their finite dimensional 
distributions. As a consequence of Theorem 3.1 therein, we have for any Borel set 
$A\subset (0,\infty)^k$,
\begin{equation}\label{eq:general}
\mathbb{P}(Z\in A)=\int 1_{\{z\in A\}}\exp\{-V(z)\}\sum_{\tau\in\mathscr{P}_k} \prod_{j=1}^{|\tau|}\int 1_{\{u_j<z_{\tau_j^c}\}}\Lambda(\mathrm{d}z_{\tau_j},\mathrm{d}u_j).
\end{equation}
where $z_{\tau_j}$ and $z_{\tau_j^c}$ denote the restrictions of the vector $z$ to
components in $\tau_j$ and $\tau_j^c=\{1,\ldots,k\}\setminus\tau_j$, respectively,
and the inner integrals are with respect to $u_j\in [0,\infty)^{|\tau_j^c|}$, 
$j=1,\ldots,|\tau|$. For instance, in the bivariate case $k=2$, the sum has only
two terms corresponding to $\tau=\{1,2\}$ and $\tau=\{\{1\},\{2\}\}$ and 
Eq.~\eqref{eq:general} reads 
\begin{align*}
\mathbb{P}(Z\in A)=&\int 1_{\{z\in A\}}\exp\{-V(z)\}\Lambda(\mathrm{d}z_1,\mathrm{d}z_2)\\
&\quad + \int 1_{\{z\in A\}}\exp\{-V(z)\} 1_{\{u_1<z_{2}\}}1_{\{u_2<z_{1}\}}\Lambda(\mathrm{d}z_{1},\mathrm{d}u_1)\Lambda(\mathrm{d}u_{2},\mathrm{d}z_2).
\end{align*}
Note that formula~\eqref{eq:general} is slightly different from that in Theorem 3.1
in \cite{DEM13} where a disintegration of $\Lambda$ into marginal and conditional 
distributions is used. The proof of Proposition~\ref{prop1} now consists in 
checking whether formula \eqref{eq:general} defines an absolutely continuous 
distribution or not. 
\medskip\noindent

We first prove that condition \textit{ii)} implies \textit{i)}. Assuming that
$\Lambda$ has a density $\lambda_I$ on each face $E_I^k$, we can plug this density 
into \eqref{eq:general} and, for any Borel set $A\subset (0,\infty)^k$, we obtain
\begin{align*}
\mathbb{P}(&Z\in A)\\
&= \int 1_{\{z\in A\}} \exp(-V(z))\sum_{\tau\in\mathscr{P}_k} \prod_{j=1}^{|\tau|} \sum_{\emptyset \neq   I\subset\{1,\ldots,k\}} \int 1_{\{u_{j}<z_{\tau_j^c}\}}\lambda_I(z_{\tau_j},u_{j})\mu_I(\mathrm{d}z_{\tau_j},\mathrm{d}u_{j})\\
&= \int 1_{\{z\in A\}}  \exp(-V(z))\sum_{\tau\in\mathscr{P}_k} \prod_{j=1}^{|\tau|} \sum_{\tau_j \subset  I\subset\{1,\ldots,k\}} \int 1_{\{u_{j}<z_{\tau_j^c}\}}\lambda_I(z_{\tau_j},u_{j})\mathrm{d}z_{\tau_j}\mathrm{d}u_{j,I}\delta_0(\mathrm{d}_{u_{j,I^c}})\\
&= \int 1_{\{z\in A\}} \left\{ \exp(-V(z))\sum_{\tau\in\mathscr{P}_k} \prod_{j=1}^{|\tau|} \sum_{\tau_j \subset  I\subset\{1,\ldots,k\}} \int 1_{\{u_{j}<z_{\tau_j^c}\}}\lambda_I(z_{\tau_j},u_{j})\mathrm{d}u_{j,I}\delta_0(\mathrm{d}_{u_{j,I^c}})\right\}\, \mathrm{d}z.
\end{align*}
For the second equality, note that the terms with $I$ not containing $\tau_j$ have 
a null contribution to the sum since then the components of $z$ in $\tau_j^c\cap I$
are set to $0$ while $z\in A\subset (0,\infty)^k$. This shows that $Z$ has density
\[
f(z)= \exp(-V(z))\sum_{\tau\in\mathscr{P}_k} \prod_{j=1}^{|\tau|} \sum_{\tau_j \subset  I\subset\{1,\ldots,k\}} \int 1_{\{u_{j}<z_{\tau_j^c}\}}\lambda_I(z_{\tau_j},u_{j})\mathrm{d}u_{j,I}\delta_0(\mathrm{d}_{u_{j,I^c}}),
\]
which corresponds to Eq.~\eqref{eq:llh_Z} with $-\partial_{\tau_j}V(z)$ given by
Eq.~\eqref{eq:Vtau}. Furthermore, Eq.~\eqref{eq:V} together with condition 
\textit{ii)} imply Eq.~\eqref{eq:V2}.

We next prove that if condition \textit{ii)} is not satisfied, then $Z$ has no
density. Assume that there is $I\subset \{1,\ldots,k\}$ and $A_I\subset E_I^k$ 
such that $\mu_I(A_I)=0$ and $\Lambda(A_I)>0$. Without loss of generality, we may
suppose that $I=\{1,\ldots,p\}$ and 
$A_I\subset [\varepsilon, \infty)^p\times\{0\}^{k-p}$, for some 
$p\in\{1,\dots, k\}$. Let $A$ be the set of all vectors $z\in E^k$ such that 
$(z_1,\ldots,z_p,0,\ldots,0)\in A_I$. The condition $\mu_I(A_I)=0$ implies
that $A$ has Lebesgue measure $0$. We will show that $\mathbb{P}(Z\in A)>0$ so 
that $Z$ has no density. Instead of using Eq.~\eqref{eq:general}, it is easier 
to reason on the Poisson point process representation of $Z$. Let $\Pi$ be a 
Poisson point process on $E^k$ with intensity $\Lambda$ such that $Z$ has the
same distribution as the componentwise maximum of the points of $\Pi$. Consider
the event
 \[
\left\{ \Pi \mbox{ has exactly one point in } A_I \mbox{ and all the other points are in } [0,\varepsilon)^p\times [0,\infty)^{k-p}\right\}.
\]
This event is equal to
\[
\left\{\Pi(A_I)=1 \mbox{ and } \Pi\left([\varepsilon,\infty)^p\times [0,\infty)^{k-p}\setminus A_I \right)=0\right\},
\]
and it has probability
\begin{align*}
\Lambda(A_I)&\exp\{-\Lambda(A_I)\}\exp\left\{-\Lambda([\varepsilon,\infty)^p\times [0,\infty)^{k-p}\setminus A_I)\right\}\\
&=\Lambda(A_I)\exp\left\{-\Lambda([\varepsilon,\infty)^p\times [0,\infty)^{k-p})\right\}>0.
\end{align*}
Furthermore, on this event, we have $Z\in A$ because the $I$-components of $Z$ are given exactly by the unique 
point in $A_I$ since all the other points have lower $I$-components. It follows that  $\mathbb{P}(Z\in A)>0$ and 
that $Z$ has no density.

Finally, the equivalence of conditions \textit{ii)} and \textit{iii)} follows from
the homogeneity property of $\Lambda$ and from the factorization of the exponent 
measure given in Eq.~\eqref{eq:H}. It follows that $\Lambda$ has a density on all 
sub-faces of $E^k$ if and only if $H$ has a density on all sub-faces of $S^{k-1}$. 
Furthermore the homogeneity property implies that the densities $\lambda_I$ and 
$h_I$ are related by Eq.~\eqref{eq:lambda}.
\end{proof}

For illustration of the above result, let us consider the case $k=2$. 
Prop.~\ref{prop1} states that a simple max-stable vector $Z = (Z_1,Z_2)$ admits
a density if and only if its angular measure has a density in the interior of 
$S^1$, and it possibly has point masses on the two vertices of $S_1$. It thus 
follows that, for instance, the asymmetric logistic distribution 
\citep{taw1988a} always possesses a density. On the other hand, for two 
independent standard Fr\'echet distributed random variables $X_1$ and $X_2$, 
the simple max-stable vector $Z = (X_1, \max\{X_1,X_2\}/2)$ does not have a 
density since its angular measure has a point mass in the interior of $S^1$.

\section{Asymptotic results} \label{sec:asymptotics}

In this section, we will provide the main results of this paper on the
asymptotic properties of likelihood-based estimators in multivariate extremes. 
Throughout this section, we consider independent observations $Z^{(1)},\ldots,Z^{(n)}$ 
stemming from a parametric family of $k$-dimensional max-stable distributions 
$\{f_\theta, \, \theta \in \Theta\}$ with parameter space $\Theta \subset \mathbb{R}^D$ 
and maximum likelihood estimators $\hat\theta_n^{\mathrm{mle}}$ based on the 
likelihood $\prod_{i=1}^n f_\theta(Z^{(i)})$. Under natural assumptions on the 
parametric family we show asymptotic normality and efficiency of the estimator as
$n \to \infty$. For independent observations $Z^{(1)},\ldots,Z^{(n)}$ with 
distribution $f_{\theta_0}$, where $\theta_0 \in \mathrm{int}(\Theta)$ is in the
interior of the parameter space, we will show that 
\begin{equation}\label{eq:asymptotic-normality} 
\sqrt{n} (\hat\theta_n^{\mathrm{mle}} - \theta_0) \stackrel{d}\longrightarrow \mathcal{N}(0,I_{\theta_0}^{-1}),\quad\mbox{as } n\to\infty,
\end{equation}
where $\stackrel{d}\longrightarrow$ denotes convergence in distribution and 
$I_{\theta_0}$ is the Fisher information matrix,
  $$ I_{\theta_0} = \int \left\{\partial_\theta \log f_{\theta_0}(z) \right\}
             \left\{\partial_\theta \log f_{\theta_0}(z) \right\}^T f_{\theta_0}(z)\,\mathrm{d}z,$$
which is assumed to be  non-singular.

Our proof makes use of the theory in \citet{vdV98}, where a key tool is the notion 
of differentiability in quadratic mean, which we discuss in Subsection 
\ref{sec:quad}. The asymptotic behavior of the maximum likelihood estimator is 
then discussed in Section \ref{sec:MLE}.

\subsection{Differentiability in quadratic mean}\label{sec:quad}

For parametric statistical models, the asymptotic theory of likelihood-based 
estimators relies strongly on the following regularity property, called 
differentiability in quadratic mean.
\begin{definition} 
Let $\theta_0\in \mathrm{int}(\Theta)$ be in the interior of $\Theta$. The 
parametric statistical model $\{f_\theta, \, \theta\in\Theta\}$ is differentiable
in quadratic mean at $\theta_0$ if there exists a measurable function 
$\dot\ell_{\theta_0}:\mathbb{R}^k\to\mathbb{R}^D$ such that
\begin{equation*}
\int_{\mathbb{R}^k}\left( \sqrt{f_{\theta_0+h}(z)}-\sqrt{f_{\theta_0}(z)}-\frac{1}{2}h^T \dot\ell_{\theta_0}(z)\sqrt{f_{\theta_0}(z)} \right)^2\mathrm{d}z=o(\|h\|^2),\quad h\to 0.
\end{equation*}
\end{definition}

Note that if the parametric statistical model $\{f_\theta, \, \theta\in\Theta\}$
is differentiable in quadratic mean at~$\theta_0$ and the mapping 
$\theta \mapsto \sqrt{f_\theta(z)}$ is differentiable at $\theta_0$ for 
$f_{\theta_0}$-almost every $z$, we have $\dot\ell_{\theta_0}(z) = \partial_\theta \log f_{\theta_0}(z)$.

Differentiability in quadratic mean ensures that the score $\dot\ell_{\theta_0}$
defines a centered random variable with finite variance equal to the Fisher 
information matrix, that is
\[
\int \dot\ell_{\theta_0}(z) f_{\theta_0}(z)\mathrm{d}z=0\quad\mbox{and}\quad \int \dot\ell_{\theta_0}(z)\dot\ell_{\theta_0}(z)^T f_{\theta_0}(z)\mathrm{d}z=I_{\theta_0}.
\]
Furthermore, the likelihood allows for an asymptotic expansion resulting in the 
so-called local asymptotic normality of the model. More precisely, for 
$Z^{(1)},\ldots,Z^{(n)}$ independent with distribution $f_{\theta_0}$, we have
\[
\log \prod_{i=1}^n \frac{f_{\theta_0+h/\sqrt{n}}}{f_{\theta_0}}(Z^{(i)})=\frac{1}{\sqrt{n}}\sum_{i=1}^n h^T\dot\ell_{\theta_0}(Z^{(i)})-\frac{1}{2}h^TI_{\theta_0}h+o_{P}(1),
\]
where $o_P(1)$ is a remainder term that converges in probability to zero if 
$h\to 0$. Local asymptotic normality is fundamental in the study of the asymptotic
properties of maximum likelihood estimator; see Chapter 7 and Theorem 7.2 in 
\citet{vdV98} for more details. 

Due to the complex structure of max-stable models, their differentiability 
in quadratic mean is difficult to analyze. Even in the one-dimensional
case, \citet{BS16a} have proved only recently that the generalized extreme 
value distribution, where $\theta$ consists of the location, the scale and 
the shape parameter, is differentiable in quadratic mean at $\theta_0$ if and 
only if the shape is greater than $-1/2$. 
To our best knowledge, no general results are available in a multivariate 
setting. We focus here on the case of simple multivariate max-stable 
distributions, where the term simple means that the margins are normalized to
standard Fr\'echet distributions. In Propositions \ref{prop:DQM1} and 
\ref{prop:DQM2} below, we provide natural conditions on the exponent functions 
$V_\theta$ and on the angular densities $h_\theta$, respectively, that ensure 
differentiability in quadratic mean.

\begin{proposition}\label{prop:DQM1}
Let $\{f_\theta,\,\theta\in\Theta\}$ be a parametric family of densities of 
multivariate simple max-stable distributions. Denote by $V_\theta$ the 
exponent function associated to $f_\theta$. Let $\theta_0\in \mathrm{int}(\Theta)$
and assume that there is a neighborhood $\Theta_0$ of $\theta_0$ such that the
following conditions hold: 
\begin{itemize}
\item[A1)] the density support $S_\theta=\{z\in(0,\infty)^k:\ f_\theta(z)>0\}$
           does not depend on $\theta\in\Theta_0$; 
\item[A2)] the derivative $(\theta,z) \in \Theta_0 \times (0,\infty)^k \mapsto \partial_\theta \partial_{\{1,\ldots,k\}} V_\theta(z)$
           exists and is continuous;
\item[A3)] for all $\theta\in\Theta_0$, $z\in(0,\infty)^k$ with $\|z\|=1$ and 
           $\tau_i\subset\{1,\ldots,k\}$,
\[
\|\partial_\theta \log V_\theta(z)\|_\infty \leq c(z)\quad \mbox{and}\quad \|\partial_\theta \log \partial_{\tau_i}V_{\theta}(z)\|_\infty \leq c(z)
\]
where $c(z)$ is a dominating function of the form 
\[
c(z)=  A\sum_{i=1}^k z_i^{-\alpha},\quad A>0,\ \alpha \in [0,1/2)
\]
and $\|\cdot\|$ is an arbitrary norm on $\mathbb{R}^k$.
\end{itemize}
Then, the model $\{f_\theta,\,\theta\in\Theta\}$ is differentiable in quadratic 
mean at $\theta_0$ with score function $\dot\ell_{\theta_0}=\partial_\theta \log f_{\theta_0}$.
\end{proposition}
For all the standard examples discussed in Section~\ref{sec:exam}, the density 
$f_\theta$ is positive on $(0,\infty)^k$ so that assumption A1) is not 
restrictive. Assumption A2) is quite natural in view of 
Equation~\eqref{eq:llh_Z}. Assumption A3) is a technical assumption that ensures 
various integrability properties in the proof of the differentiability in 
quadratic mean. The form of the dominating function $c(z)$ is general enough to
cover all popular models discussed in Section~\ref{sec:exam}.

\begin{proposition}\label{prop:DQM2}
Let $\{f_\theta,\,\theta \in \Theta\}$ be a parametric family of densities of 
multivariate simple max-stable distributions. Denote by $h_\theta$ the angular 
density associated to $f_\theta$. Let $\theta_0\in \mathrm{int}(\Theta)$ and 
assume that there is a neighborhood $\Theta_0$ of $\theta_0$ such that the 
following conditions hold for all nonempty subsets $I \subset \{1,\ldots,k\}$:
\begin{itemize}
\item[B1)] the density support $S_\theta=\{z\in(0,\infty)^k:\ f_\theta(z)>0\}$
           does not depend on $\theta\in\Theta_0$; 
\item[B2)] the derivative $(\theta,w) \in \Theta_0 \times S^{k-1}_I \mapsto \partial_\theta h_{\theta,I}(w)$ exists and is continuous;
\item[B3)] there are  $B^-,B^+\geq 0$ and $0<\beta_i^+<\beta_i^-<(1+\varepsilon)\beta_i^+$ with 
$0<2\varepsilon<\{\sum_{i=1}^k\beta_i^-\}^{-1}$ such that
\[
h_{\theta,I}(w) \geq B^- \prod_{i=1}^k w_i^{-1+\beta_i^-} 
\quad \mbox{and}\quad  
\|\partial_\theta h_{\theta,I}(w)\|_\infty \leq B^+ \prod_{i=1}^k w_i^{-1+\beta_i^+}
\]
for all $\theta\in\Theta_0$ and $w\in S^{k-1}_I$.
\end{itemize}
Then, the model $\{f_\theta,\,\theta\in\Theta\}$ is differentiable in quadratic 
mean at $\theta_0$ with score function 
$\dot\ell_\theta=\partial_\theta \log f_{\theta_0}$.
\end{proposition}
Applications of these general criteria to specific models are postponed to Section~\ref{sec:exam}.

\subsection{Asymptotic theory for the maximum likeliood estimator} \label{sec:MLE}
The asymptotic normality of a consistent sequence of maximum likelihood 
estimators follows from the differentiability in quadratic mean of the statistical
model, a Lipschitz property for the log-likelihood and the non-singularity of the
Fisher information matrix. More precisely, let us recall for future reference the 
following theorem.

\begin{theorem}[\citet{vdV98}, Theorem 5.39]\label{thm:vdV-mle}
Let $\{f_\theta,\theta\in\Theta\}$ be differentiable in quadratic mean at 
$\theta_0\in \mathrm{int}(\Theta)$ with non-singular Fisher information matrix 
$I_{\theta_0}$. Suppose that the following Lipschitz condition is satisfied: 
for all $\theta_1,\theta_2\in\Theta$ 
\begin{equation}\label{eq:lipschitz}
\left|\log f_{\theta_2}(z)-\log f_{\theta_1}(z)\right|\leq \tilde\ell(z)\,\|\theta_2-\theta_1\|_\infty,
\end{equation}
for some measurable function $\tilde\ell$ such that 
$\int \tilde\ell(z)^2f_{\theta_0}(z)\,\mathrm{d}z<\infty$. Then, any consistent
sequence of maximum likelihood estimators $\hat\theta_n^{\mathrm{mle}}$ is
asymptotically normal and efficient as~$n\to\infty$, that is, it satisfies 
Equation~\eqref{eq:asymptotic-normality}.
\end{theorem}

In order to apply this theorem in the framework of max-stable distributions, we
mainly need to check the differentiability in quadratic mean of the statistical 
model and the Lipschitz condition. Differentiability in quadratic mean is considered
in the previous subsection and proved under conditions A1)--A3) or B1)--B3) in
Propositions~\ref{prop:DQM1} and~\ref{prop:DQM2}, respectively. Interestingly, 
these conditions imply also the Lipschitz condition \eqref{eq:lipschitz} and we 
obtain the following result.

\begin{theorem} \label{prop:mle}
 Let $\{f_\theta, \theta \in \Theta\}$ be a parametric family of identifiable
 densities of multivariate simple max-stable distributions. Let 
 $\theta_0 \in \mathrm{int}(\Theta)$ and assume that conditions A1)--A3) or 
 B1)--B3) hold and  that the Fisher information matrix $I_{\theta_0}$ is 
 non-singular. Then there exists a sequence $\hat \theta_n^{\mathrm{mle}}$ of 
 local maxima of the log-likelihood that is asymptotically normal and efficient,
 that is, it satisfies Equation~\eqref{eq:asymptotic-normality}, as~$n\to\infty$. 
\end{theorem}
\begin{proof}
 From the proof of Proposition \ref{prop:DQM2} it follows that conditions 
 B1)--B3) imply conditions A1)--A3) with $\|\cdot\|=\|\cdot\|_\infty$. Thus, 
 without loss of generality, we may  assume that conditions A1)--A3) hold, and
 by Proposition \ref{prop:DQM1}, that the model is differentiable in quadratic 
 mean at $\theta_0$. 

We next prove that conditions A1)--A3) imply the Lipschitz condition~\eqref{eq:lipschitz}
for all $\theta_1,\theta_2\in\Theta_0$, where $\Theta_0$ is the neighborhood of $\theta_0$
where conditions A1)--A3) hold. Indeed, by Equation A2), the derivative 
$\partial_\theta \log f_\theta(z) = \partial_\theta f_\theta(z) / f_\theta(z)$ exists
for every $\theta \in \Theta_0$ and $z \in (0,\infty)^k$ and, by
 Equation \eqref{eq:scorebound}, satisfies  $\| \partial_\theta \log f_\theta(z)\|_\infty \leq \tilde \ell(z)$
 with $\tilde \ell(z) = \left( k + \sum_{i=1}^k z_i^{-1}\right) c\left(z/\|z\|\right)$.
 Consequently, we obtain
 \begin{align*}
 | \log f_{\theta_2}(z) - \log f_{\theta_1}(z) | ={}& \left| \int_{0}^{1} \left\{\partial_\theta \log f_{\theta_1 + t (\theta_2-\theta_1)}(z)\right\}^T (\theta_2-\theta_1) \, \mathrm{d} t \right|\\
   \leq{}& \int_{0}^{1} k^2 \, \| \partial_\theta \log f_{\theta_1 + t (\theta_2-\theta_1)}(z) \|_\infty \, \|\theta_2-\theta_1\|_\infty \mathrm{d} t  \\
   \leq{}& k^2 \, \tilde \ell(z) \, \|\theta_2 - \theta_1\|_\infty.
 \end{align*}
The finiteness of the integral $\int \tilde \ell(z)^2 f_{\theta}(z) \mathrm{d}z$ 
follows as in the proof of Proposition~\ref{prop:DQM1}. This proves the Lipschitz 
condition~\eqref{eq:lipschitz}.

 Now, choose $\delta >0$ such that the compact set  $B_\delta(\theta_0) = \{\theta\in\Theta:\, \|\theta-\theta_0\| \leq \delta\}$
 is contained in $\Theta_0$. Then, the Lipschitz condition entails 
$|\log f_\theta(x) - \log f_{\theta_0}(z)| \leq \delta \tilde \ell(z)$
 for all $\theta \in B_\delta(\theta_0)$, $z \in (0,\infty)^k$, with 
 $\mathbb{E}_{\theta_0} |\tilde \ell(Z)| < \infty$. By Theorem 17 in
 \citet{ferguson-96}, the sequence $\hat \theta_n^{(\delta)}$
 of maximum likelihood estimators
 $$ \hat \theta_n^{(\delta)} = \mathrm{argmax}_{\theta \in B_\delta(\theta_0)} \left\{\frac 1 n \sum\nolimits_{i=1}^n \log f_\theta\left(Z^{(i)}\right) \right\} $$
 restricted to $B_\delta(\theta_0)$ converges almost surely to $\theta_0$, as
 $n\to\infty$. This implies that $\hat \theta_n^{(\delta)}$ is eventually in 
 the interior of $B_\delta(\theta_0)$ and hence a local maximum of the 
 log-likelihood. Thus, there exists a strongly consistent sequence 
 $\{\hat \theta_n\}_{n \in \mathbb{N}}$ of local maxima of the log-likelihood
 function, which is asymptotically normal and efficient according to 
 Theorem~\ref{thm:vdV-mle}.
\end{proof}

For simplicity, in the sequel, we will always denote by the maximum likelihood estimator 
$\hat \theta_n^{\mathrm{mle}}$ the sequence of local maxima of the log-likelihood 
as defined in Theorem \ref{prop:mle}. For the global maximum, the following remark 
provides technical conditions that ensure asymptotic normality.

\begin{remark}
Proposition \ref{prop:mle} entails the existence of an asymptotically
normal sequence of local maxima of the log-likelihood function. A similar
result for the global maximum of the log-likelihood function, can be shown
under some additional assumptions.
Assume the parameter space $\Theta$ can be partitioned into a compact set $\Theta_0$ and 
a finite number of sets $L_1,\ldots,L_m$ such that
 \begin{itemize}
  \item $\theta_0 \in \mathrm{int}(\theta_0)$;
  \item conditions A1)--A3) or B1)--B3) hold locally in a neighborhood of any point of $\Theta_0$;
  \item $ \displaystyle \mathbb{E}_{\theta_0}\Big(\sup_{\theta \in L_j} \log \left\{f_\theta(Z) / f_{\theta_0}(Z)\right\}\Big) < 0$ for all $j=1,\ldots,m$.
 \end{itemize}
 Then, any sequence $\hat \theta_n^{mle}$ of \emph{global} maxima of the log-likelihood 
 is asymptotically normal and efficient as $n\to\infty$. 
 {Note that the third condition always holds true if $\Theta$ is compact.} 
\end{remark}

\section{Examples} \label{sec:exam}

In this section, we consider some popular parametric families of max-stable distributions,
which all admit densities as a simple consequence of Prop.~\ref{prop1}.
We show that the maximum likelihood estimator for the respective parameters 
in these models is asymptotically normal.

To this end, we show for each model that it satisfies either conditions 
A1)--A3) or conditions B1)--B3). Proposition \ref{prop:DQM1} and Proposition
\ref{prop:DQM2}, respectively, then entail that the model is differentiable 
in quadratic mean and, provided that the Fischer information matrix is non-singular,
there exists an asymptotically normal sequence of local maxima of the log-likelihood
function; see Theorem \ref{prop:mle}. For the models under consideration, in view
of the respective forms of the exponent function and the angular density, it is 
straight-forward to show that conditions A1)--A2) and B1)--B2) are satisfied. Thus, 
the proofs given in the appendix focus on the verification of condition A3) and B3), 
respectively.

\subsection{Logistic model}\label{sec:log}
The family of logistic max-stable distributions is defined by the exponent functions
$$V_\theta(z)=\left(z_1^{-1/\theta}+\cdots+z_k^{-1/\theta}\right)^\theta, \quad z\in(0,\infty)^k,$$
where the parameter $\theta\in (0,1)$ interpolates between complete dependence for $\theta \to 0$
and independence for $\theta \to 1$.

\begin{proposition} \label{prop:cond-logistic}
 The logistic model is differentiable in quadratic mean at any $\theta_0 \in (0,1)$.
 Furthermore, the maximum likelihood estimator $\hat \theta_n^{mle}$ is asymptotically
 normal and efficient as $n\to\infty$.
\end{proposition}

\begin{proof}
  The logistic model satisfies conditions A1)--A3) by Lemma \ref{lem:log} in the Appendix,
  and by Prop.~\ref{prop:DQM1} this implies differentiability in quadratic mean.
  Furthermore, Theorem~2 yields the asymptotic normality of the maximum likelihood
  estimator.
\end{proof}

\subsection{Dirichlet model}\label{sec:dir}

The max-stable family of $k$-dimensional Dirichlet distributions with parameters
$\alpha_1,\dots, \alpha_k>0$ is characterized by the angular densities
\begin{equation} \label{diri-spec}
  h(w) =  \frac 1 k \frac{\Gamma(1+\sum_{i=1}^k \alpha_i)}{(\sum_{i=1}^k \alpha_i w_i)^{k+1}} \prod_{i=1}^k \frac{\alpha_i}{\Gamma(\alpha_i)}  
  \left(\frac{\alpha_i w_i}{\sum_{j=1}^k \alpha_j w_j}\right)^{\alpha_i-1}, \quad w \in S^{k-1},
\end{equation}
and it has no mass on lower-dimensional faces of $S^{k-1}$ \citep{col1991}.

\begin{proposition} \label{prop:cond-dirichlet}
 The Dirichlet model is differentiable in quadratic mean at any parameter vector 
 $\theta_0=(\alpha_1,\ldots,\alpha_k)\in \Theta=(0,\infty)^k$. Furthermore, the 
 maximum likelihood estimator $\hat \theta_n^{mle}$ is asymptotically normal
 and efficient as $n \to \infty$.
\end{proposition}

\begin{proof}
 By Lemma~\ref{lem:dir}, the Dirichlet model satisfies conditions B1)--B3).
 Differentiability in quadratic mean then follows from Prop.~\ref{prop:DQM2},
 and Theorem~\ref{prop:mle} further implies asymptotic normality and efficiency
 of the maximum likelihood estimator.
\end{proof}

\subsection{Extremal-$t$ model and Schlather process}\label{sec:ext}
 
The family of extremal-$t$ distributions is parameterized by $\nu>0$ and a
positive definite correlation matrix $\Sigma$. It can be characterized by its 
angular densities $h_I$ on each face $S_I^{k-1}$, $I\subset \{1,\dots,k\}$. On 
the interior of $S^{k-1}$, it is given by
\begin{align}\label{extr_spec}
  h_{\{1,\dots,k\}}(w) = C(\Sigma,\nu) \cdot \left\{\left(w^{1/\nu}\right)^\top \Sigma^{-1} w^{1/\nu}\right\}^{-\frac{k+\nu} 2} \cdot \prod_{i=1}^k w_i^{\frac{1-\nu}{\nu}}, \quad w \in S_{\{1,\ldots,k\}}^{k-1},
\end{align}
where $$C(\Sigma,\nu) = \pi^{(1-k)/2} \Gamma\left(\frac{\nu+1} 2\right)^{-1} \Gamma\left(\frac{k + \nu} 2\right) \nu^{-k+1} \det(\Sigma)^{-1/2}.$$ 
It can be deduced from \citet{R13} that the angular densities on the 
lower-dimensional faces are of the same form.

\begin{proposition} \label{prop:extremalt}
 The extremal-$t$ model is differentiable in quadratic mean at any
 $\theta_0=(\Sigma,\nu)$ with $\Sigma$ being a positive definite correlation
 matrix and $\nu>0$. Furthermore, the maximum likelihood estimator
 $\hat \theta_n^{mle}$ is asymptotically normal and efficient as $n\to\infty$. 
\end{proposition}
\begin{proof}
  The extremal-$t$ model satisfies conditions B1)--B3) by Lemma~\ref{lem:ext} 
  in the Appendix. By Prop.~\ref{prop:DQM2} this implies differentiability in
  quadratic mean and, furthermore, asymptotic normality and efficiency of the 
  maximum likelihood estimator follow from Theorem~\ref{prop:mle}.
\end{proof}

In spatial extremes, the class of extremal-$t$ distributions appears as the 
finite dimensional distributions of the extremal-$t$ process 
$\{Z(t), \ t \in \mathbb{R}^d\}$, which is max-stable and stationary 
\citep{opi2013}. It is parameterized by a correlation function 
$\rho: \mathbb{R}^d \to [-1,1]$ and a single value $\nu>0$ via the relation
 $$ \Sigma = \{\rho(t_i-t_j)\}_{1 \leq i,j \leq k} $$
where $\nu$ and $\Sigma$ are the parameters of the spectral density of the 
distribution of the random vector $(Z(t_1),\ldots,Z(t_n))$ as in
\eqref{extr_spec}. The special case $\nu=1$ corresponds to the extremal
Gaussian process \citep{sch2002}, also called Schlather process.

\begin{corollary}\label{cor:ext}
  Let $Z$ be a Schlather process on $\mathbb R^d$ with correlation function $\rho$
  coming from the parametric family
  $$\rho(h) = \exp(-\|h\|_2^\alpha/s), \quad (s,\alpha) \in \Theta = (0,\infty)\times (0,2].$$
  Suppose that $Z$ is observed at pairwise distinct locations $t_1,\dots, t_k \in \mathbb R^d$ 
  such that not all pairs of locations have the same Euclidean distance. 
  Then, the maximum likelihood estimator of $\theta=(s,\alpha)$ is asymptotically normal.
\end{corollary}
\begin{proof}
 Suppose that $\|t_1-t_2\|_2 \neq \|t_2-t_3\|_2$ and observe that the mapping
 $\Psi:\Theta \to \Psi(\Theta)$, $\theta = (s,\alpha) \mapsto \{\rho_{ij}\}_{1 \leq i, j \leq k}  = \{\exp(- \|t_i -t_j\|_2^\alpha/s)\}_{1 \leq i, j \leq k}$
 is continuously  differentiable. Since
 $$ \alpha = \frac{\log\{\log\rho_{12}\} - \log\{\log\rho_{23}\}}{\log \|t_1-t_2\|_2 - \log\|t_2-t_3\|_2},\qquad s= - \frac{\|t_1-t_2\|_2^\alpha}{\log\rho_{12}},$$
 the same holds true for the inverse mapping $\Psi^{-1}$. Thus, from 
 Lemma~\ref{lem:ext} in the Appendix it follows that the Schlather process
 satisfies conditions B1)--B3) as well.
 Hence, Prop.~\ref{prop:DQM2} and Thm.~\ref{prop:mle} imply differentiability
 in quadratic mean of the model and asymptotic normality and efficiency of 
 the maximum likelihood estimator.
\end{proof}

\subsection{H\"usler--Reiss model and Brown--Resnick process} \label{sec:hr}

The H\"usler--Reiss distribution is parameterized by a strictly 
conditionally negative definite matrix $\Lambda = \{\lambda^2_{i,j}\}_{1\leq i,j\leq k}$,
and it can be characterized by its exponent function 
\begin{align}\label{exp_HR}
  V_\Lambda(z_1,\dots, z_k) = \sum_{i=1}^k  z_i^{-1} \Phi_{k-1}\left( 2\lambda^2_{i,-i}  + \log(z_{-i}/z_i); R^{(i)}\right), \quad z\in(0,\infty)^k,
\end{align}
\citep[cf.,][]{hue1989, nik2009}, where for $i=1,\dots, k$, the strictly positive
definite matrix $R^{(i)}$ has $(j,m)$th entry
$2(\lambda^2_{i,j} + \lambda^2_{i,m} - \lambda^2_{j,m})$, $j,m\neq i$ \citep[Lemma 3.2.1]{ber1984}.
Here and in the sequel, for $p\in\mathbb N$, $\Phi_{p}(\cdot, R)$ and 
$\varphi_{p}(\cdot, R)$ denote the $p$-dimensional normal distribution function 
and density with covariance matrix $R$, respectively.

\begin{proposition}\label{prop:cond-hr}
 The H\"usler--Reiss model is differentiable in quadratic mean at any 
 $\theta_0 = \Lambda$ with $\Lambda$ being a strictly conditionally negative 
 definite matrix. Furthermore, the maximum likelihood estimator 
 $\hat \theta_n^{mle}$ is asymptotically normal and efficient as $n\to\infty$.
\end{proposition}
\begin{proof}
   The H\"usler--Reiss model satisfies conditions A1)--A3) by Lemma~\ref{lem:hr} in the Appendix.
   By Prop.~\ref{prop:DQM2} this implies differentiability in quadratic mean and,
   furthermore, asymptotic normality and efficiency of the maximum likelihood estimator 
   follow from Theorem~\ref{prop:mle}.
\end{proof}

H\"usler--Reiss distributions are the finite dimensional distributions of the max-stable Brown--Resnick
process, a popular class in spatial extreme value statistics that are parameterized
by conditionally negative definite variograms \citep{bro1977,kab2009}. The most common
parametric class are the fractional variograms, which we consider in the following corollary.
\begin{corollary}
  Consider a Brown--Resnick process on $\mathbb R^d$ with variogram
  coming from the parametric family 
  $$\gamma(h) = \|h\|_2^\alpha/\lambda, \quad (\lambda, \alpha) \in \Theta = (0,\infty)\times (0,2).$$
  Suppose that 
  the process is observed on a finite set of locations $t_1,\dots, t_m\in \mathbb R^d$
  such that the pairwise Euclidean distances are not all equal.
  Then the maximum likelihood estimator of $\theta=(\lambda,\alpha)$ is
  asymptotically normal.
\end{corollary}
\begin{proof}
 Suppose that $\|t_1-t_2\|_2 \neq \|t_2-t_3\|_2$ and observe that the mapping
 $\Psi:\Theta \to \Psi(\Theta)$, $\theta = (\lambda,\alpha) \mapsto \{\gamma_{ij}\}_{1 \leq i,j \leq k}  = \{ \|t_i - t_j\|_2^\alpha/\lambda\}_{1 \leq i,j \leq k}$ 
 is continuously differentiable. Since
 $$ \alpha = \frac{\log\gamma_{12} - \log \gamma_{23}}{\log \|t_1-t_2\|_2 - \log \|t_2-t_3\|_2},\qquad \lambda= \frac{\|t_i-t_j\|_2^\alpha}{\gamma_{ij}},$$
 the same holds true for the inverse mapping $\Psi^{-1}$. Thus, from 
 Lemma~\ref{lem:hr} in the Appendix it follows that the Brown--Resnick 
 process satisfies conditions A1)--A3) as well.
 Hence, Prop.~\ref{prop:DQM2} and Thm.~\ref{prop:mle} imply differentiability
 in quadratic mean of the model and asymptotic normality and efficiency of 
 the maximum likelihood estimator.
\end{proof}

\bibliography{MaxStable_Bayesian}
\bibliographystyle{abbrvnat}

\appendix

\section{Proofs postponed from Section \ref{sec:asymptotics}} \label{sec:proof1}

\subsection{Proof of Prop.\ \ref{prop:DQM1}}

According to Lemma 7.6 in \citet{vdV98}, it suffices to verify the following 
two conditions:
\begin{itemize}
\item[i)] $\theta\mapsto \sqrt{f_\theta(z)}$ is continuously differentiable for every $z$;
\item[ii)] the application $\theta\mapsto I_\theta=\int \left(\frac{\partial_\theta f_\theta}{f_\theta}\right)
           \left(\frac{\partial_\theta f_\theta}{f_\theta}\right)^T f_\theta(z)\,\mathrm{d}z$ 
            is well defined and continuous.
\end{itemize}
Formula \eqref{eq:llh_Z} for the likelihood together with assumptions A1) and 
A2) imply point i). Indeed, $\theta\mapsto f_\theta(z)$ is either positive and 
continuously differentiable or identically equal to $0$; in both cases 
$\theta\mapsto \sqrt{f_\theta(z)}$ is continuously differentiable. We next 
focus on point ii) and prove first that the integral is well defined. From~\eqref{eq:llh_Z} 
we derive the upper bound
\begin{equation} \label{eq:bound-partial-deriv}
\|\partial_\theta f_\theta(z)\|_\infty \leq \left( \|\partial_\theta V_\theta(z)\|_\infty + 
\max_{\tau\in\mathscr{P}_k}\sum_{j=1}^{|\tau|} \left\|\frac{\partial_\theta \partial_{\tau_j} V_{\theta}(z)}{ \partial_{\tau_j}V_{\theta}(z)}\right\|_\infty \right)f_\theta(z).
\end{equation}
Recall that the exponent function $V_\theta$ is homogeneous of order $-1$, and,
inductively, it can be verified that the derivative 
$\partial_{\tau_i} V_\theta$ is homogeneous of order $-1 -|\tau_i|$, for all 
$\tau_i \subset \{1,\ldots,k\}$. Consequently, we have 
$\partial_\theta \log V_\theta(z) = \partial_\theta \log V_\theta(uz)$ and
$\partial_\theta \log \partial_{\tau_i}V_{\theta}(z) = \partial_\theta \log \partial_{\tau_i}V_{\theta}(uz)$
for all $u>0$, $z \in (0,\infty)^k$ and $\tau_i \subset \{1,\ldots,k\}$.
Thus, assumption A3) can be reformulated as
\begin{align*}
 \|\partial_\theta \log V_\theta(z)\|_\infty ={}& \left\|\partial_\theta \log V_\theta\left(\frac z {\|z\|}\right)\right\|_\infty \leq c\left(\frac{z}{\|z\|}\right) \\
 \mbox{and} \quad \|\partial_\theta \log \partial_{\tau_i} V_\theta(z)\|_\infty ={}& \left\|\partial_\theta \log \partial_{\tau_i} V_\theta\left(\frac z {\|z\|}\right)\right\|_\infty \leq c\left(\frac{z}{\|z\|}\right)
\end{align*}
for all $z \in (0,\infty)^k$. It follows that the right hand side of 
\eqref{eq:bound-partial-deriv} is upper bounded by 
$V_\theta(z)c(z/\|z\|)+kc(z/\|z\|)$ and, thanks to the general bound 
$V_\theta(z)\leq \sum_{i=1}^k z_i^{-1}$, we deduce
\begin{equation}\label{eq:scorebound}
\left\|\frac{\partial_\theta f_\theta(z)}{f_\theta(z)}\right\|_\infty \leq  \left( k+\sum_{i=1}^k z_i^{-1}\right) c\left(\frac{z}{\|z\|}\right).
\end{equation}
Using H\"older's inequality with $p,q\geq 1$ such that $1/p+1/q=1$, we obtain
\begin{eqnarray}
& &\int  \left\| \left(\frac{\partial _\theta f_\theta}{f_\theta }\right) \left(\frac{\partial _\theta f_\theta}{f_\theta }\right)^T\right\|_\infty  f_\theta(z)\mathrm{d}z 
  \leq  \int \left( k+\sum_{i=1}^k z_i^{-1}\right)^2 c\left(\frac{z}{\|z\|}\right)^2 f_\theta(z)\mathrm{d}z \nonumber\\ 
&\leq& \left(\int \left( k+\sum_{i=1}^k z_i^{-1}\right)^{2p} f_\theta(z)\mathrm{d}z\right)^{1/p} \left(\int c\left(\frac{z}{\|z\|}\right)^{2q} f_\theta(z)\mathrm{d}z\right)^{1/q}. \label{eq:scorebound2}
\end{eqnarray}
We prove that the two integrals in \eqref{eq:scorebound2} are finite. For the
first integral, we use the fact that under $f_\theta(z)\mathrm{d}z$, $z_i^{-1}$
has an  exponential distribution with mean $1$ and hence finite moment of all 
orders. The sum $ k+\sum_{i=1}^k z_i^{-1}$ has thus a finite moment of order 
$2p$ for all $p\geq 1$. Choosing $p$ large for the first integral in the right
hand side of \eqref{eq:scorebound2}, we can choose $q>1$ close to $1$ in the 
second integral. In the upper bound $c(z/\|z\|)=A\sum_{i=1}^k\|z\|^{\alpha}z_i^{-\alpha}$,
one can chose the norm $\|z\|=\|z\|_\infty$ as all norms on $\mathbb{R}^k$ are
equivalent. Then $\|z\|_\infty$ has a Fr\'echet distribution with shape parameter
1 and a scale parameter between 1 and $k$ (depending on $\theta$) under 
$f_\theta(z)\mathrm{d}z$ and $\|z\|_\infty^\alpha$ has finite moment of order 
$r<1/\alpha$; similarly $z_i^{-1}$ has a standard exponential distribution so that 
$z_i^{-\alpha}$ has finite moment of any order $r>0$. By H\"older's inequality,
the product $\|z\|_\infty^{\alpha}z_i^{-\alpha}$ has finite moments of order 
$r<1/\alpha$. Since $1/\alpha>2$, one can chose $q>1$ close to $1$ so that 
$\|z\|_\infty^{\alpha}z_i^{-\alpha}$ has a finite moment of order $2q$. Then the 
second integral in \eqref{eq:scorebound2} is finite. We deduce that 
$\theta\mapsto I_\theta$ is well defined. 

We next consider continuity. For all continuous and bounded function 
$F:\mathbb{R}^k\to [-M,M]$, $M>0$, the mapping  
\begin{equation}\label{eq:conv}
\theta \mapsto \int F\left(\frac{\partial _\theta f_\theta (z)}{f_\theta (z)}\right) f_\theta(z) \,\mathrm{d}z
\end{equation}
is continuous. Indeed, setting 
$\theta\mapsto \tilde F(\theta,z)=F \left\{\partial _\theta f_\theta (z) / f_\theta (z)\right\}$,
we have
\begin{eqnarray*}
&&\left|\int \tilde F(\theta_n,z)f_{\theta_n}(z) \,\mathrm{d}z-\int \tilde F(\theta,z)f_\theta(z) \,\mathrm{d}z\right|\\
&\leq & \int |\tilde F(\theta_n,z)\{f_{\theta_n}(z)-f_{\theta}(z)\}|\,\mathrm{d}z+
\int |\tilde F(\theta_n,z)-\tilde F(\theta,z)|f_{\theta}(z)\,\mathrm{d}z \\
&\leq & M\int |f_{\theta_n}(z)-f_{\theta}(z)|\,\mathrm{d}z+\int |\tilde F(\theta_n,z)-\tilde F(\theta,z)|f_{\theta}(z)\,\mathrm{d}z. 
\end{eqnarray*}
Applying Scheff\'e's lemma for the first term and Lebesgue's dominated 
convergence theorem for the second, both terms converge to zero if $\theta_n \to \theta$ for $n\to\infty$. This proves
the continuity of \eqref{eq:conv} and implies that the score defined by 
\[
S_\theta=\frac{\partial _\theta f_\theta (Z_\theta)}{f_\theta (Z_\theta)}\quad \mbox{with } Z_\theta\sim f_\theta(z)\mathrm{d}z,
\]
is continuous in distribution, i.e., $S_{\theta_n}\stackrel{d}\to S_\theta$ as
$\theta_n\to\theta$. Continuity of the information matrix 
$I_\theta=\mathbb{E} \left(S_\theta S_\theta^T\right)$ follows then from the uniform 
integrability condition 
\[
\sup_{\theta\in\Theta_0} \mathbb{E}\left[ \|S_\theta\|_\infty^{2+\varepsilon} \right]<\infty\quad \mbox{with } \varepsilon>0,
\] 
which is  a consequence of the bound \eqref{eq:scorebound}, since for all 
$\theta\in\Theta_0$,
\[
\mathbb{E}\left[ \|S_\theta\|_\infty^{2+\varepsilon} \right]
\leq \int\left(k+\sum_{i=1}^k z_i^{-1}\right)^{2+\varepsilon} c\left(\frac{z}{\|z\|_\infty}\right)^{2+\varepsilon}  f_\theta(z)\,\mathrm{d}z. 
\]
The right hand side can be upper bounded independently of $\theta\in\Theta_0$, 
similarly as in \eqref{eq:scorebound2} where the bound does not depend on 
$\theta$ as $\|z\|_\infty$ has a Fr\'echet distribution with shape parameter 1
and scale parameter bounded between 1 and $k$ and $z_i^{-1}$ has a unit
exponential distribution under $f_\theta(z)\mathrm{d}z$.

\subsection{Proof of Prop.\ \ref{prop:DQM2}}
We show that assumptions B1)--B3) imply assumptions A1)--A3). As assumptions
A1) and B1) are identical, we just verify that conditions B2) and B3)
together imply conditions A2) and A3).

For the sake of simplicity, we assume that the angular density is positive on 
$S^{k-1}_I$, $I=\{1,\ldots,k\}$ and vanishes on the lower-dimensional faces
$S^{k-1}_I$, $I\subsetneq \{1,\ldots,k\}$. Then, using formulas \eqref{eq:V},
\eqref{eq:Vtau} and \eqref{eq:lambda},  we have
\begin{eqnarray*}
\partial_\theta \partial_{\tau_j}V_{\theta}(z)
&=& -\partial_\theta \int_{(0,z_{\tau_j^c})} \|(z_{\tau_j},u)\|_1^{-k-1} h_\theta\left(\frac{(z_{\tau_j},u)}{\|(z_{\tau_j},u)\|_1}\right)\, \mathrm{d}u.
\end{eqnarray*}
A change of derivation and integration gives 
\begin{eqnarray*}
\partial_\theta \partial_{\tau_j}V_{\theta}(z)
&=&  - \int_{(0,z_{\tau_j^c})} \|(z_{\tau_j},u)\|_1^{-k-1} 
\partial_\theta h_\theta\left(\frac{(z_{\tau_j},u)}{\|(z_{\tau_j},u)\|_1}\right)\, \mathrm{d}u,
\end{eqnarray*}
which is allowed since by B3)
\[
 \|(z_{\tau_j},u)\|_1^{-k-1} \left\|\partial_\theta h_\theta\left(\frac{(z_{\tau_j},u)}{\|(z_{\tau_j},u)\|_1}\right)\right\|_\infty
\leq  B^+\|(z_{\tau_j},u)\|_1^{-1-\sum_{i=1}^k \beta_i^+}  \prod_{i \in \tau_j}z_i^{-1+\beta^+_i} \prod_{i \notin \tau_j} u_i^{-1+\beta^+_i}  
\]
and 
\[
\int_{(0,z_{\tau_j^c})} \|(z_{\tau_j},u)\|_1^{-1-\sum_{i=1}^k \beta_i^+} \prod_{i \notin \tau_j} u_i^{-1+\beta_i^+} \, \mathrm{d}u
\leq \|z_{\tau_j}\|_1^{-1-\sum_{i=1}^k \beta_i^+}\prod_{i \notin \tau_j} \frac{z_i^{\beta_i^+}}{\beta_i^+} < \infty.
\]
To prove condition A3), we will make use of the following lemma.
\begin{lemma}\label{lem1}
Let $\beta_1,\ldots,\beta_k>0$ and $\beta=\sum_{i=1}^k\beta_i$. There are
constants $C^-,C^+>0$ such that, for all $z\in (0,\infty)^k$,
\[
\int_{(0,z_{\tau_j^c})} \|(z_{\tau_j},u)\|_1^{-1-\beta} \prod_{i \notin \tau_j} u_i^{-1+\beta_i} \mathrm{d}u  
\leq C^+\|z_{\tau_j}\|_\infty^{-1-\sum_{i\in\tau_j}\beta_i}.
\]
and
\[
\int_{(0,z_{\tau_j^c})} \|(z_{\tau_j},u)\|_1^{-1-\beta}\prod_{i \notin \tau_j} u_i^{-1+\beta_i} \mathrm{d}u 
\geq C^-\|z_{\tau_j}\|_\infty^{-1-\sum_{i\in\tau_j}\beta_i} \prod_{i\in\tau_j^c}\left(\frac{z_i}{\|z_{\tau_j}\|_\infty}\wedge 1\right)^{\beta_i}
\]
\end{lemma}
\begin{proof} 
As all norms are equivalent on $\mathbb R^k$, we can replace the norm $\|\cdot\|_1$ 
by the sup-norm $\|\cdot\|_\infty$. Since $\|(z_{\tau_j},u)\|_\infty$ can be either 
equal to $\|z_{\tau_j}\|_\infty$ (if $u_i<\|z_{\tau_j}\|_\infty$ for all 
$i\in\tau_j^c$) or equal to $u_{i_0}$ for some $i_0\in\tau_j^c$ (if $u_{i_0}>\|z_{\tau_j}\|_\infty$
and $u_i<u_{i_0}$ for all $i\in\tau_j^c\setminus\{i_0\}$), we obtain
\begin{eqnarray*}
&&\int_{(0,z_{\tau_j^c})} \|(z_{\tau_j},u)\|_\infty^{-1-\beta} \prod_{i \in \tau_j^c} u_i^{-1+\beta_i} \mathrm{d}u\\
&=& \int_{(0,z_{\tau_j^c}\wedge \|z_{\tau_j}\|_\infty)} \|z_{\tau_j}\|_\infty^{-1-\beta} \prod_{i \in \tau_j^c} u_i^{-1+\beta_i} \mathrm{d}u\\
&&\quad +\sum_{i_0\in\tau_j^c}\int_{\|z_{\tau_j}\|_\infty<u_{i_0}<z_{i_0},\ u_i<u_{i_0}\wedge z_i, i \in \tau_j^c \setminus\{i_0\}} u_{i_0}^{-2-\beta+\beta_{i_0}} 
\prod_{i \in \tau_j^c\setminus\{i_0\}} u_i^{-1+\beta_i} \mathrm{d}u\\
&=& \|z_{\tau_j}\|_\infty^{-1-\beta}\prod_{i\in\tau_j^c}\frac{(z_i\wedge \|z_{\tau_j}\|_\infty)^{\beta_i}}{\beta_i}\\
&&\quad +\sum_{i_0\in\tau_j^c}\int_{\|z_{\tau_j}\|_\infty<u_{i_0}<z_{i_0}} u_{i_0}^{-2-\beta+\beta_{i_0}}\prod_{i \in \tau_j^c\setminus\{i_0\}} \frac{(u_{i_0}\wedge z_i)^{\beta_i}}{\beta_i} \mathrm{d}u_{i_0}
\end{eqnarray*}
The lower bounds corresponds simply to the first term. For the upper bound, we note
that 
\[
\|z_{\tau_j}\|_\infty^{-1-\beta}\prod_{i\in\tau_j^c}\frac{(z_i\wedge \|z_{\tau_j}\|_\infty)^{\beta_i}}{\beta_i}\leq \left(\prod_{i\in\tau_j^c}\frac{1}{\beta_i}\right)\|z_{\tau_j}\|_\infty^{-1-\sum_{i\in\tau_j}\beta_i}
\]
and, for $i_0 \in \tau_j^c$,
\begin{eqnarray*}
&&\int_{\|z_{\tau_j}\|_\infty<u_{i_0}<z_{i_0}} u_{i_0}^{-2-\beta+\beta_{i_0}}\prod_{i \in \tau_j^c\setminus\{i_0\}} \frac{(u_{i_0}\wedge z_i)^{\beta_i}}{\beta_i} \mathrm{d}u_{i_0}\\
&\leq& \left(\prod_{i\in\tau_j^c\setminus\{i_0\}}\frac{1}{\beta_i}\right)\int_{\|z_{\tau_j}\|_\infty<u_{i_0}}u_{i_0}^{-2-\sum_{i\in\tau_j}\beta_i}\mathrm{d}u_{i_0}\\
&\leq &\left(\prod_{i\in\tau_j^c\setminus\{i_0\}}\frac{1}{\beta_i}\right)\left(1+\sum_{i\in\tau_j}\beta_i\right)^{-1}\|z_{\tau_j}\|_\infty^{-1-\sum_{i\in\tau_j}\beta_i}.
\end{eqnarray*}
The upper bound follows.
\end{proof}
We now prove that condition A3) is satisfied. The upper bound from condition B3)
entails
\begin{eqnarray}
 \|\partial_\theta \partial_{\tau_j}V_{\theta}(z)\|_\infty
&\leq &  \int_{(0,z_{\tau_j^c})} \|(z_{\tau_j},u)\|_1^{-k-1}\left\|\partial_\theta h_\theta\left(\frac{(z_{\tau_j},u)}{\|(z_{\tau_j},u)\|_1}\right)\right\|_\infty\, \mathrm{d}u\nonumber\\
&\leq &B^+ \prod_{i \in \tau_j}z_i^{-1+\beta^+_i} \int_{(0,z_{\tau_j^c})} \|(z_{\tau_j},u)\|_1^{-1-\sum_{i=1}^k \beta_i^+}   \prod_{i \notin \tau_j} u_i^{-1+\beta^+_i} \mathrm{d}u.\label{prop8:proof1}
\end{eqnarray}
We upper bound the integral on the right hand side using H\"older's inequality 
and Lemma~\ref{lem1}. Let $\varepsilon>0$ be as in B3), $q=1+\varepsilon$, $p>1$
such that  $1/p+1/q=1$ and $\tilde \beta_i=\beta_i^+-\beta_i^-/q$. Note that 
$q>\beta_i^-/\beta_i^+$ implies $\tilde\beta_i>0$. With the notation
$\beta^+=\sum_{i=1}^k \beta_i^+$, $\beta^-=\sum_{i=1}^k \beta_i^-$ and 
$\tilde\beta=\sum_{i=1}^k \tilde \beta_i$, we have
\begin{eqnarray}
 &&\int_{(0,z_{\tau_j^c})} \|(z_{\tau_j},u)\|_1^{-1- \beta^+}  \prod_{i \notin \tau_j} u_i^{-1+\beta^+_i} \mathrm{d}u\nonumber\\
&=&\int_{(0,z_{\tau_j^c})} \left(\|(z_{\tau_j},u)\|_1^{(-1-\beta^-)/q}  
     \prod_{i \notin \tau_j} u_i^{(-1+\beta^-_i)/q} \right)\left(\|(z_{\tau_j},u)\|_1^{-1+1/q-\tilde\beta}  \prod_{i \notin \tau_j} u_i^{-1+1/q+\tilde\beta_i} \right) \mathrm{d}u\nonumber\\
&\leq & \left(\int_{(0,z_{\tau_j^c})}\|(z_{\tau_j},u)\|_1^{-1-\beta^-}  \prod_{i \notin \tau_j} u_i^{-1+\beta^-_i}  \mathrm{d}u \right)^{1/q}\left(\int_{(0,z_{\tau_j^c})} \|(z_{\tau_j},u)\|_1^{-1-p\tilde\beta} 
  \prod_{i \notin \tau_j} u_i^{-1+p\tilde\beta_i}  \mathrm{d}u \right)^{1/p}.\label{prop8:proof2}
\end{eqnarray}
The lower bound from condition B3) entails
\begin{eqnarray}
 |\partial_{\tau_j}V_{\theta}(z)|
&= &  \int_{(0,z_{\tau_j^c})} \|(z_{\tau_j},u)\|_1^{-k-1} h_\theta\left(\frac{(z_{\tau_j},u)}{\|(z_{\tau_j},u)\|_1}\right) \mathrm{d}u\nonumber\\
&\geq &B^- \prod_{i \in \tau_j}z_i^{-1+\beta^-_i} \int_{(0,z_{\tau_j^c})} \|(z_{\tau_j},u)\|_1^{-1-\sum_{i=1}^k \beta_i^-}   \prod_{i \notin \tau_j} u_i^{-1+\beta^-_i} \mathrm{d}u.\label{prop8:proof3}
\end{eqnarray}
Equations \eqref{prop8:proof1},\eqref{prop8:proof2} and \eqref{prop8:proof3} imply 
the upper bound for the ratio
\begin{eqnarray*}
&&\left\|\frac{\partial_\theta \partial_{\tau_j}V_{\theta}(z)}{\partial_{\tau_j}V_{\theta}(z)}\right\|_\infty
\leq  \frac{B^+}{B^-} \prod_{i \in \tau_j}z_i^{-(\beta^-_i-\beta_i^+)}\frac{ \left(\int_{(0,z_{\tau_j^c})} \|(z_{\tau_j},u)\|_1^{-1-p\tilde\beta}  \prod_{i \notin \tau_j} u_i^{-1+p\tilde\beta_i}  \mathrm{d}u \right)^{1/p}}{\left(\int_{(0,z_{\tau_j^c})} \|(z_{\tau_j},u)\|_1^{-1-\sum_{i=1}^k \beta_i^-}   \prod_{i \notin \tau_j} u_i^{-1+\beta^-_i} \mathrm{d}u \right)^{1-1/q}}.
\end{eqnarray*}
Then, Lemma \ref{lem1} implies
\begin{eqnarray*}
&&\left\|\frac{\partial_\theta \partial_{\tau_j}V_{\theta}(z)}{\partial_{\tau_j}V_{\theta}(z)}\right\|_\infty\\
&\leq & \frac{B^+(C^+)^{1/p}}{B^-(C^-)^{1-1/q}} \prod_{i \in \tau_j}z_i^{-(\beta^-_i-\beta_i^+)}\frac{\|z_{\tau_j}\|_\infty^{-(1+\sum_{i\in\tau_j}p\tilde{\beta}_i)/p}}{\|z_{\tau_j}\|_\infty^{-(1+\sum_{i\in\tau_j}\beta_i^-)(1-1/q)} \prod_{i\in\tau_j^c}\left(\frac{z_i}{\|z_{\tau_j}\|_\infty}\wedge 1\right)^{\beta_i^-(1-1/q)}}\\
&= & \frac{B^+(C^+)^{1/p}}{B^-(C^-)^{1-1/q}}\prod_{i \in \tau_j}\left(\frac{z_i}{\|z_{\tau_j}\|_\infty}\right)^{-(\beta^-_i-\beta_i^+)}\prod_{i\in\tau_j^c}\left(\frac{z_i}{\|z_{\tau_j}\|_\infty}\wedge 1\right)^{-(1-1/q)\beta_i^-}.
\end{eqnarray*}
For $\|z\|_\infty=1$, we use the bounds $\|z_{\tau_j}\|_\infty \leq 1$, $z_i \leq 1$ and $0<\beta_i^- - \beta_i^+ < (1-1/q)\beta_i^-$
and obtain
\[
\left\|\frac{\partial_\theta \partial_{\tau_j}V_{\theta}(z)}{\partial_{\tau_j}V_{\theta}(z)}\right\|
\leq A \prod_{i=1}^k\left( \frac{1}{z_i}\right)^{(1-1/q)\beta_i^-}\leq A \max_{i=1}^k \left( \frac{1}{z_i}\right)^\alpha \leq  A\sum_{i=1}^k\left( \frac{1}{z_i}\right)^\alpha
\]
with $A > 0$, $\alpha=(1-1/q)\sum_{i=1}^k \beta_i^-$ and $q=1+\varepsilon$. 
The assumption $2\varepsilon<\{\sum_{i=1}^k \beta_i^-\}^{-1}$ in B3) ensures that 
$\alpha\in (0,1/2)$ which proves A3) for $\|\cdot\|_\infty$.

\section{Proofs postponed from Section \ref{sec:exam}} \label{sec:proof2}

\subsection{Logistic model}
\begin{lemma}\label{lem:log}
  The logistic model introduced in Section \ref{sec:log} satisfies conditions A1)--A3).
\end{lemma}
\begin{proof}
  To verify assumption A3) for the logistic model with $\theta\in (0,1)$, we compute
  \begin{eqnarray}\label{eq:log}
   \partial_\theta \log V_\theta(z)&=& \log\left(\sum_{i=1}^k z_i^{-1/\theta}\right)+\frac{1}{\theta}\frac{\sum_{i=1}^k z_i^{-1/\theta}\log z_i}{\sum_{i=1}^k z_i^{-1/\theta}}.
  \end{eqnarray}
  For the first term we get for $\|z\| = 1$
  \begin{eqnarray*}
   \left | \log\left(\sum_{i=1}^k z_i^{-1/\theta}\right) \right | \leq \log\left\{ k \left(\min_{i=1,\dots,k} z_i\right)^{-1/\theta}\right\}  \leq C \left(\min_{i=1,\dots,k} z_i\right)^{-\epsilon/\theta} \leq  C\sum_{i=1}^k z_i^{-\epsilon/\theta},
  \end{eqnarray*}
  for any $\epsilon >0 $ and some $C>0$. For the second term we have
  \begin{eqnarray*}
   \left | \frac{1}{\theta}\frac{\sum_{i=1}^k z_i^{-1/\theta}\log z_i}{\sum_{i=1}^k z_i^{-1/\theta}}  \right | \leq  C_1(\theta) \sum_{i=1}^k |\log z_i| \leq C_2(\theta) \sum_{i=1}^k z_i^{-\epsilon},
  \end{eqnarray*}
  for positive continuous functions $C_1$ and $C_2$ of $\theta$.
  Next, consider
  \begin{eqnarray*}
   - \partial_{\tau_j} V_{\theta}(z)&=& \prod_{i=1}^{|\tau_j| -1} \left(\frac{i}{\theta} -1 \right)\left(\sum_{i=1}^k z_i^{-1/\theta}\right)^{\theta - |\tau_j|}
   \prod_{i\in\tau_j} z_i^{-1/\theta -1}, 
  \end{eqnarray*}
  and thus
  \begin{eqnarray*} 
   \partial_\theta \log \{- \partial_{\tau_j} V_{\theta}(z)\}&=& \log\left(\sum_{i=1}^k z_i^{-1/\theta}\right)+ \frac{\theta-|\tau_j|}{{\theta^2}}\frac{\sum_{i=1}^k z_i^{-1/\theta}\log z_i}{\sum_{i=1}^k z_i^{-1/\theta}}\\
   && + \frac{1}{\theta^2}\sum_{i\in\tau_j}\log z_i - \sum_{i=1}^{|\tau_j| -1} \frac{i}{i\theta - \theta^2}.
  \end{eqnarray*}
  All terms can be treated similarly as for \eqref{eq:log}. Choosing $\epsilon > 0$
  sufficiently small yields~A3).
\end{proof}

\subsection{Dirichlet model}

\begin{lemma}\label{lem:dir}
  The Dirichlet model introduced in Section \ref{sec:dir} satisfies conditions B1)--B3).
\end{lemma}
\begin{proof}
  To verify condition B1) from Proposition we recall the angular density of the 
  Dirichlet distribution in \eqref{diri-spec}. From this, it is directly seen that
  for some positive and continuous function $C(\alpha_1,\ldots,\alpha_n)$
  $$h(w) \geq C(\alpha_1,\ldots,\alpha_n) \prod_{i=1}^k w_i^{-1+\alpha_i},$$
  since for any $(\alpha_1,\dots, \alpha_k)\in\Theta_0$ and $(w_1,\dots, w_k)\in S^{k-1}$
  the term $\sum_{j=1}^k \alpha_j w_j$ is uniformly bounded from below and above.
  Differentiating the function $h$ with respect to $\alpha_m$, for 
  $m\in\{1,\dots, k\}$,  we obtain the upper bound
  \begin{align*}
    |\partial_{\alpha_m} h(w)| &\leq C_1(\alpha_1,\ldots,\alpha_n) \prod_{i=1}^k w_i^{\alpha_i-1} + C_1(\alpha_1,\ldots,\alpha_n) w_m^{\alpha_m-1} |\log w_m|\prod_{i\neq n} w_i^{\alpha_i-1}\\
    &\leq C_1 \prod_{i=1}^k w_i^{-1+\alpha_i - \delta},
  \end{align*}
  for any $\delta > 0$ and a positive and continuous function $C_1(\alpha_1,\ldots,\alpha_n)>0$.
  Thus we get for some $\epsilon>0$ that
  $$0 < \beta_i^+ = \alpha_i - \delta < \alpha_i = \beta_i^- < (1+\epsilon) \beta_i^+,$$
  and when $\delta$ goes to 0 we can choose $\epsilon$ arbitrarily small.
\end{proof}

\subsection{Extremal-$t$ model and Schlather process}

\begin{lemma}\label{lem:ext}
  The extremal-$t$ model introduced in Section \ref{sec:ext} satisfies conditions 
  B1)--B3).
\end{lemma}
\begin{proof}
   As the angular density $h_I$ is of the same form \eqref{extr_spec} for every
   face $S_I^{k-1}$, $I \subset \{1,\ldots,k\}$, we restrict ourselves to the 
   case $I=\{1,\ldots,k\}$. The verification works analogously for the other faces.

   The angular density $h = h_{\{1,\ldots,k\}}$ in \eqref{extr_spec} is strictly 
   positive and continuously differentiable both with respect to $w$ and 
   $\theta=(\Sigma,\nu)$, and it satisfies conditions B1) and B2). As $\Sigma$ is 
   positive definite, we have that both the minimal and the maximal eigenvalue
   $\sigma_{\min}$ and $\sigma_{\max}$ of $\Sigma^{-1}$, respectively, are positive.
   Consequently,
   $$ \sigma_{\min} \sum_{i=1}^k w_i^{2/\nu} \leq \left(w^{1/\nu}\right)^\top \Sigma^{-1} w^{1/\nu} \leq \sigma_{\max} \sum_{i=1}^k w_i^{2/\nu}.$$
   With $\sum_{i=1}^k w_i = 1$, we have that $\sum_{i=1}^n w_i^{2/\nu}$ can be 
   bounded from below and above and, thus, there exist positive continuous 
   functions $C_1(\Sigma,\nu)$ and $C_2(\Sigma,\nu)$ such that
   $$ C_1(\Sigma,\nu) \cdot \prod_{i=1}^k w_i^{\frac{1-\nu}{\nu}} \leq h(w) \leq C_2(\Sigma,\nu) \cdot \prod_{i=1}^k w_i^{\frac{1-\nu}{\nu}}.$$
   Further, for the partial derivatives w.r.t.\ the components of $\Sigma$, we
   obtain that
   \begin{align*}
    \partial_\Sigma h(w) ={}& \left\{\partial_\Sigma C(\Sigma,\nu)\right\} \cdot \left\{\left(w^{1/\nu}\right)^\top \Sigma^{-1} w^{1/\nu}\right\}^{-\frac{k+\nu}2} \prod_{i=1}^k w_i^{\frac{1-\nu}{\nu}}\\
                      & - C(\Sigma,\nu) \cdot \frac{k+\nu} 2 \left(w^{1/\nu}\right)^\top \left(\partial_\Sigma \Sigma^{-1}\right) w^{1/\nu} \cdot
                         \left\{\left(w^{1/\nu}\right)^\top \Sigma^{-1} w^{1/\nu}\right\}^{- \frac{k+\nu} 2 - 1} \prod_{i=1}^k w_i^{\frac{1-\nu}{\nu}}.
   \end{align*}
   As $\left|(w^{1/\nu})^\top \left(\partial_\Sigma \Sigma^{-1}\right) w^{1/\nu}\right|$
   can be bounded by the product of the largest eigenvalue of 
   $\partial_\Sigma \Sigma^{-1}$ and $\sum_{i=1}^k w_i^{2/\nu}$, and the latter 
   can again be bounded by a constant, we obtain that
   $$ \left|\partial_\Sigma h(w)\right| \leq C_3(\Sigma,\nu) \cdot \prod_{i=1}^k w_i^{\frac{1-\nu}{\nu}}$$
   for some positive continuous function $C_3(\Sigma,\nu)$.
   
    For the partial derivative of the angular density w.r.t.\ $\nu$, we get
   \begin{align*}
   \partial_\nu h(w) ={}& \left\{\partial_\nu C(\Sigma,\nu)\right\} \cdot \left\{\left(w^{1/\nu}\right)^\top \Sigma^{-1} w^{1/\nu}\right\}^{-\frac{k+\nu} 2} \prod_{i=1}^k w_i^{\frac{1-\nu}{\nu}}\\
               & + \frac{C(\Sigma,\nu)}{2\nu^2} \cdot f\left(\Sigma,\nu,w\right) \cdot \left\{\left(w^{1/\nu}\right)^\top \Sigma^{-1} w^{1/\nu}\right\}^{-\frac{k+\nu} 2} \prod_{i=1}^k w_i^{\frac{1-\nu}{\nu}}\\
               & + \frac{C(\Sigma,\nu)}{\nu^2} \cdot \left\{\left(w^{1/\nu}\right)^\top \Sigma^{-1} w^{1/\nu}\right\}^{-\frac{k+\nu} 2} \cdot \log\left(\prod_{i=1}^k w_i\right) 
                           \cdot \prod_{i=1}^k w_i^{\frac{1-\nu}{\nu}}
   \end{align*}
   where
   $$ f(\Sigma,\nu,w) = - \frac 1 2 \log\left\{\left(w^{1/\nu}\right)^\top \Sigma^{-1} w^{1/\nu}\right\} 
                        + \frac{k + \nu}{2 \nu^2} \cdot \frac{\sum_{i=1}^k \sum_{j=1}^k \Sigma^{-1}_{ij} \cdot \left(w_i w_j\right)^{1/\nu} \cdot \log(w_i w_j)}{\left(w^{1/\nu}\right)^\top \Sigma^{-1} w^{1/\nu}}. $$
   As the function $x \mapsto x^{1/\nu} \log(x)$ is bounded on $(0,1]$ and
   $$ \left|\log\left\{\left(w^{1/\nu}\right)^\top \Sigma^{-1} w^{1/\nu}\right\}\right| \leq \max\left\{\left|\log\left(\sigma_{\min} \sum_{i=1}^k w_i^{2/\nu}\right)\right|, 
                                                                                           \left|\log\left(\sigma_{\max} \sum_{i=1}^k w_i^{2/\nu}\right)\right|\right),$$
   we obtain that $f(\Sigma,\nu,\cdot)$ can be bounded from above, i.e.\ there
   exists a positive continuous function $C_4(\Sigma,\nu)$ such that
   $$ |\partial_\nu h(w)| \leq{} C_4(\Sigma,\nu) \cdot \left\{1-\log\left(\prod_{i=1}^k w_i\right)\right\} \cdot \prod_{i=1}^k w_i^{\frac{1-\nu}{\nu}}
                          \leq{} \frac{C_4(\Sigma,\nu)}\varepsilon \prod_{i=1}^k w_i^{\frac{1-\nu}{\nu}-\varepsilon}$$
   for all $\varepsilon >0$. Condition B3) follows from the fact that $C_1$, $C_2$, $C_3$ and $C_4$ are continuous.
\end{proof}   

\subsection{H\"usler--Reiss model and Brown--Resnick process}

\begin{lemma}\label{lem:hr}
  The H\"usler--Reiss  model introduced in Section \ref{sec:hr} satisfies 
  conditions A1)--A3).
\end{lemma}
\begin{proof} 
   We check condition A3) of Proposition \ref{prop:DQM1}. We recall the form of 
   the exponent function of the $k$-variate H\"usler--Reiss distribution in 
   \eqref{exp_HR} that is parameterized by the conditionally negative definite
   parameter matrix $\Lambda = \{\lambda^2_{i,j}\}_{1 \leq i,j \leq k}$ with 
   $\lambda^2_{i,i} = 0$, $i\in\{1,\dots, k\}$. Taking the derivative with respect
   to $\lambda^2_{j,m}$, $j\neq m$, yields
  \begin{align}\label{fraction_eq}
    \partial_{\lambda^2_{j,m}} \log V_\Lambda(z_1,\dots, z_k) = \frac{\sum_{i=1}^k z_i^{-1} \partial_{\lambda^2_{j,m}} \Phi_{k-1}\left( 2\lambda^2_{i,-i}  + \log(z_{-i}/z_i); R^{(i)}\right)}{V_\Lambda(z_1,\dots, z_k)}
  \end{align}
  Using the fact that the density $\varphi_{k-1}(\cdot; R^{(i)})$ can be bounded
  locally uniformly in $\lambda_{j,m}^2$ by an integrable function, we may 
  exchange the order of integration and differentiation when differentiating
  $$  \Phi_{k-1}\left( 2\lambda^2_{i,-i}  + \log(z_{-i}/z_i) \right)
  = \int_{-\infty}^{\log(z_{-i}/z_i)} \varphi_{k-1}\left(x_{-i} + 2\lambda^2_{i,-i}; R^{(i)}\right) \mathrm d x_{-i}.$$
  For $i\notin \{j,m\}$, we compute
  \begin{align*}
  &\partial_{\lambda^2_{j,m}} \Phi_{k-1}\left( 2\lambda^2_{i,-i} +
  \log(z_{-i}/z_i); R^{(i)}\right) = \int_{-\infty}^{\log(z_{-i}/z_i)}
  \partial_{\lambda^2_{j,m}}\varphi_{k-1}\left(x_{-i} + 2\lambda^2_{i,-i};
  R^{(i)}\right) \mathrm d x_{-i}\\ & =
  \int_{-\infty}^{\log(z_{-i}/z_i)} \partial_{\lambda^2_{j,m}}
  (2\pi)^{(1-k)/2} |R^{(i)}|^{-1/2} \exp\left\{- \frac12 (x_{-i} +
  2\lambda^2_{i,-i})^\top (R^{(i)})^{-1}(x_{-i} +
  2\lambda^2_{i,-i})\right\} \mathrm d x_{-i}\\ & =
  \int_{-\infty}^{\log(z_{-i}/z_i)} \left( -\frac12 (x_{-i} +
  2\lambda^2_{i,-i})^\top \partial_{\lambda^2_{j,m}}(R^{(i)})^{-1}(x_{-i}
  + 2\lambda^2_{i,-i}) - \frac{\partial_{\lambda^2_{j,m}}
    |R^{(i)}|}{2|R^{(i)}|} \right)\\
  & \hspace{9cm} \cdot \varphi_{k-1}\left(x_{-i} +
  2\lambda^2_{i,-i}; R^{(i)}\right) \mathrm d x_{-i},
  \end{align*}
  where $|A|$ denotes the absolute value of the determinant of a matrix $A$.
  Using the eigenvalue upper bound for $(x_{-i} + 2\lambda^2_{i,-i})^\top \partial_{\lambda^2_{j,m}}(R^{(i)})^{-1}(x_{-i} + 2\lambda^2_{i,-i})$ 
  and the lower bound for $(x_{-i} + 2\lambda^2_{i,-i})^\top (R^{(i)})^{-1}(x_{-i} + 2\lambda^2_{i,-i})$, 
  where we denote the smallest eigenvalue of $(R^{(i)})^{-1}$ by 
  $\mu^{(i)}_{min}>0$, we obtain 
  \begin{align}
  \notag&\left|\partial_{\lambda^2_{j,m}} \Phi_{k-1}\left( 2\lambda^2_{i,-i}  + \log(z_{-i}/z_i); R^{(i)}\right)\right|\\
  \notag&\leq C'(\Lambda) \int_{-\infty}^{\log(z_{-i}/z_i)}  
  \left( \sum_{l\neq i} (x_{l} + 2\lambda^2_{i,l})^2 + 1  \right) \varphi_1\left\{\mu^{(i)}_{min}(x_{l} + 2\lambda^2_{i,l}); 1\right\} \mathrm d x_{-i}\\
  \label{boundjm}&\leq C''(\Lambda), 
  \end{align}  
  since the normal distribution has moments of all order. Here $C'(\Lambda)$ and
  $C''(\Lambda)$  are continuous functions. If $i \in\{j,m\}$, the calculations
  are similar with some additional linear terms in $x_{-i}$ appearing in the 
  derivative, which are dominated by the quadratic terms. Thus, for all $j\neq m$
  and all $i=1,\dots, k$, we have the bound \eqref{boundjm}. It suffices to
  recall the general bound
  $$ V_\Lambda(z_1,\dots, z_k) \geq 1 /\min(z_1,\dots, z_k),$$
  and to multiply numerator and denominator in equation \eqref{fraction_eq} by
  $\min(z_1,\dots, z_k)$ to conclude that
  $$\| \partial_{\Lambda} \log V_\Lambda(z_1,\dots, z_k) \| \leq C(\Lambda),$$
  for some continuous function $C(\Lambda)$. 

  Suppose without loss of generality that $1 \in \tau\subset \{1,\dots, k\}$. We
  introduce the notation $\tilde \tau = \tau\setminus\{1\}$, $\tau^c = 
  \{1,\dots, k\} \setminus \tau$, and we let $R$ be the  matrix with $(j,m)$th
  entry $2(\lambda^2_{1,j} + \lambda^2_{1,m} - \lambda^2_{j,m})$, $j,m =1,\dots,k$. 
  The partial derivative of $V_\Lambda$ with respect to $\tau$ can be represented as
  \citep[cf.][]{asa2015}
  \begin{align*}
   \partial_{\tau} V_\Lambda(z_1,\dots, z_k) = \frac{1}{z_1^2\prod_{i\in \tilde\tau} z_i}
   \varphi_{|\tilde\tau|}(z^*_{\tilde\tau}; R_{\tilde\tau,\tilde\tau})
   \Phi_{|\tau^c|}(z^*_{\tau^c} - \hat \mu; \hat R),
  \end{align*}
  where $z^* = \log(z/z_1) + 2\lambda^2_{\{1,\dots, k\} ,1}$, and   
  \begin{align*}
    \hat \mu &= R_{\tau^c,\tilde\tau} R^{-1}_{\tilde\tau,\tilde\tau}  z^*_{\tilde\tau}, \qquad   \hat R = R_{\tau^c,\tau^c} - R_{\tau^c,\tilde\tau}R^{-1}_{\tilde\tau,\tilde\tau}R_{\tilde\tau,\tau^c},
  \end{align*}
  are the 
  conditional mean and covariance matrix, respectively. We now consider for $j\neq m$
  \begin{align*}
   &\partial_{\lambda^2_{j,m}} \log \partial_{\tau} V_\Lambda(z_1,\dots, z_k) =\\ 
   &-\frac12 \partial_{\lambda^2_{j,m}}  \log|R| -   \frac12 \partial_{\lambda^2_{j,m}} \left\{(z^*_{\tilde\tau})^\top R_{\tilde\tau,\tilde\tau}^{-1}z^*_{\tilde\tau}\right\}
    + \partial_{\lambda^2_{j,m}} \log \Phi_{|\tau^c|}(z^*_{\tau^c} - \hat \mu; \hat R).
  \end{align*}
  Suppose from now on that $\|z\| = 1$. The absolute value of the second summand of 
  the right hand side can be bounded, up to  a continuous function of $\Lambda$,
  by the product of the maximal eigenvalue of $\partial_{\lambda^2_{j,m}}R_{\tilde\tau,\tilde\tau}^{-1}$
  and 
  $$\|z^*_{\tilde\tau}\|_2^2 = \sum_{i\in\tilde\tau} \left\{\log(z_i/z_1) + 2\lambda^2_{i ,1}\right\}^2
  \leq C(\Lambda) \sum_{i\in\tilde\tau} \max(z_1/z_i,z_i/z_1)^\epsilon\leq C'(\Lambda) \sum_{i=1}^k {z_i}^{-\epsilon}.$$

  For the third term, we first note that, with a similar argument as above, we may 
  exchange the order of integration and differentiation. We compute
  \begin{align*}
  & \partial_{\lambda^2_{j,m}} \Phi_{|\tau^c|}(z^*_{\tau^c} - \hat \mu;\hat R)\\ 
  & =
  \int_{-\infty}^{\log(z_{\tau^c}/z_1)} \left( -\frac12 \partial_{\lambda^2_{j,m}} \left\{(x_{\tau^c} + 2\lambda^2_{\tau^c,1} - \hat \mu)^\top \hat R^{-1}(x_{\tau^c} + 2\lambda^2_{\tau^c,1} - \hat \mu)\right\} - \frac{\partial_{\lambda^2_{j,m}}
    | \hat R|}{2|\hat R|} \right)\\
  &\hspace{9cm}\cdot \varphi_{|\tau^c|}\left((x_{\tau^c} + 2\lambda^2_{\tau^c,1} - \hat \mu) ; \hat R\right) \mathrm d x_{\tau^c}.
  \end{align*} 
  We obtain the bound
  \begin{align*}
   &\left| \frac12 \partial_{\lambda^2_{j,m}} (x_{\tau^c} + 2\lambda^2_{\tau^c,1} - \hat \mu)^\top \hat R^{-1}(x_{\tau^c} + 2\lambda^2_{\tau^c,1} - \hat \mu) - \frac{\partial_{\lambda^2_{j,m}}
    | \hat R|}{2|\hat R|} \right|\\
  &\hspace{7cm}\leq C(\Lambda) \sum_{i\in \tilde\tau} \left(|\log(z_i/z_1)|^2 + 1\right) \sum_{i\in \tau^c} \left( x_i^2 + 1\right),
  \end{align*}
  since the left hand side is quadratic in both $x_{\tau^c}$ and $z^*_{\tilde\tau}$.
  Thus, we have
  \begin{align*}
  & \left| \partial_{\lambda^2_{j,m}} \Phi_{|\tau^c|}(z^*_{\tau^c} - \hat \mu;\hat R)\right|\\
  &\leq C(\Lambda) \sum_{i\in \tilde\tau} \left(|\log(z_i/z_1)|^2 + 1\right) \cdot \\
    & \qquad \cdot \sum_{i \in \tau^c} 
    \bigg[\int_{-\infty}^{z^*_{\tau^c} - \hat \mu} 
       \left( x_i - 2 \lambda^2_{\tau^c,1} +  \hat \mu\right)^2 \varphi_{|\tau^c|}(x_{\tau^c}; \hat R) \, \mathrm d x_{\tau^c} + \Phi_{|\tau^c|}(z^*_{\tau^c} - \hat \mu; \hat R)\bigg] \displaybreak[0]\\
  &\leq C(\Lambda) \sum_{i\in \tilde\tau} \left(|\log(z_i/z_1)|^2 + 1\right) |\tau^c| \cdot 
    \left[1 + 2 (2\lambda^2_{\tau^c,1} - \hat \mu)^2\right] 
        \cdot \Phi_{|\tau^c|}(z^*_{\tau^c} - \hat \mu; \hat R) \\
  & \qquad + 2 C(\Lambda) \sum_{i\in \tilde\tau} \left(|\log(z_i/z_1)|^2 + 1\right) 
       \sum_{i \in \tau^c} \int_{-\infty}^{z^*_{\tau^c} - \hat \mu} 
             x_i^2 \varphi_{|\tau^c|}(x_{\tau^c}; \hat R) \, \mathrm d x_{\tau^c}            
  \end{align*}
  To summarize, there are continuous functions $C_1(\Lambda), C_2(\Lambda) > 0$ such that
  \begin{align*}
   \left| \partial_{\lambda^2_{j,m}} \log \Phi_{|\tau^c|}(z^*_{\tau^c} - \hat \mu;\hat R)\right|& \leq{} C_1(\Lambda) \cdot |\tau^c| \cdot \sum_{i \in \tilde \tau} \left(|\log(z_i/z_1)|^4 + 1\right) \\
   & + C_2(\Lambda) \sum_{i \in \tilde \tau} \left(|\log(z_i/z_1)|^2 + 1\right) 
         \frac{\int_{-\infty}^{z^*_{\tau^c} - \hat \mu} \|x_{\tau^c}\|_2^2 \varphi_{|\tau^c|}(x_{\tau^c}; \hat R) \, \mathrm d x_{\tau^c}}
              {\int_{-\infty}^{z^*_{\tau^c} - \hat \mu} \varphi_{|\tau^c|}(x_{\tau^c}; \hat R) \, \mathrm d x_{\tau^c}}
  \end{align*}
  In the following, we will determine an upper bound for the ratio 
  \begin{equation} \label{eq:ratio-hr}
  \frac{\int_{-\infty}^{z^*_{\tau^c} - \hat \mu} \|x_{\tau^c}\|^2 \varphi_{|\tau^c|}(x_{\tau^c}; \hat R) \, \mathrm d x_{\tau^c}}
                {\int_{-\infty}^{z^*_{\tau^c} - \hat \mu} \varphi_{|\tau^c|}(x_{\tau^c}; \hat R) \, \mathrm d x_{\tau^c}}.
  \end{equation}              
  First, we consider the case that $z^*_{i} - R_{i,\tilde\tau} R^{-1}_{\tilde\tau,\tilde\tau} z^*_{\tilde\tau} > -1$ 
  for all $i \in \tau^c$. Then, \eqref{eq:ratio-hr} can be bounded by
  $\mathbb{E}(\sum_{i \in \tau^c} X_i^2) / \mathbb{P}(\max_{i \in \tau^c} X_i \leq -1)$
  where $(X_i)_{i \in \tau^c} \sim \mathcal{N}(0,\hat R)$, i.e., a positive continuous
  function of $\Lambda$. 

  Now assume that $b_{\min} = \min_{i \in \tau^c} z^*_{i} - R_{i,\tilde\tau} R^{-1}_{\tilde\tau,\tilde\tau} z^*_{\tilde\tau} < -1$.
  Then, for all $x_{\tau^c}$ in the domain of integration in \eqref{eq:ratio-hr}, we have
  \begin{align*}
  \|x_{\tau^c}\|_2^2 \leq{}& 2 \frac{b_{\min}^2}{\sigma_{\min}} \cdot \exp\left(\frac12 \frac{\sigma_{\min} }{b_{\min}^2} \|x_{\tau^c}\|^2\right)
                 {}\leq{}  2 \frac{b_{\min}^2}{\sigma_{\min}} \cdot \exp\left(\frac12 \frac{1}{b_{\min}^2} x_{\tau^c}^\top \hat R^{-1} x_{\tau^c}\right) 
  \end{align*}
  where $\sigma_{\min}$ denotes the smallest eigenvalue of $\hat R^{-1}$. Thus, we 
  obtain the bound
  \begin{align*}
  & \frac{\int_{-\infty}^{z^*_{\tau^c} - \hat \mu} \|x_{\tau^c}\|^2 \varphi_{|\tau^c|}(x_{\tau^c}; \hat R) \, \mathrm d x_{\tau^c}}
                {\int_{-\infty}^{z^*_{\tau^c} - \hat \mu} \varphi_{|\tau^c|}(x_{\tau^c}; \hat R) \, \mathrm d x_{\tau^c}}\\
  \leq{}& \frac{2 b_{\min}^2}{\sigma_{\min}} 
           \frac{\int_{-\infty}^{z^*_{\tau^c} - \hat \mu} 
                     \exp\left(- \frac 1 2 \left[1- \frac{1}{b_{\min}^2}\right] x_{\tau^c}^\top \hat R^{-1} x_{\tau^c}\right) \, \mathrm d x_{\tau^c}}
                {\int_{-\infty}^{z^*_{\tau^c} - \hat \mu} \exp\left(- \frac 1 2 x_{\tau^c}^\top \hat R^{-1} x_{\tau^c}\right)  \, \mathrm d x_{\tau^c}}\\
  ={}& \frac{2 b_{\min}^2}{\sigma_{\min}} 
           \frac{\int_{-\infty}^{z^*_{\tau^c} - \hat \mu} 
                     \exp\left(- \frac 1 2 \left[1- \frac{2}{b_{\min}^2}\right] x_{\tau^c}^\top \hat R^{-1} x_{\tau^c}\right)
                     \exp\left(- \frac 1 2 \frac{1}{b_{\min}^2} x_{\tau^c}^\top \hat R^{-1} x_{\tau^c}\right)  \, \mathrm d x_{\tau^c}}
                {\int_{-\infty}^{z^*_{\tau^c} - \hat \mu} 
                     \exp\left(- \frac 1 2 x_{\tau^c}^\top \hat R^{-1} x_{\tau^c}\right)  \, \mathrm d x_{\tau^c}} \displaybreak[0]\\
  \leq{}& \frac{2 b_{\min}^2}{\sigma_{\min}}  \left[\int_{-\infty}^{z^*_{\tau^c} - \hat \mu} 
                     \exp\left(- \frac 1 2 x_{\tau^c}^\top \hat R^{-1} x_{\tau^c}\right)  \, \mathrm d x_{\tau^c} \right]^{-2/b_{\min}^2} \\
                    & \hspace{3cm} \cdot
                     \left[\int_{-\infty}^{z^*_{\tau^c} - \hat \mu} 
                     \exp\left(- \frac 1 4 x_{\tau^c}^\top \hat R^{-1} x_{\tau^c}\right)  \, \mathrm d x_{\tau^c} \right]^{2/b_{\min}^2},       
  \end{align*}
  where we used H\"older's inequality. As $b_{\min}^2 > 1$, the last factor can be 
  bounded by a constant independent from $b_{\min}$. Noting that there is a 
  constant $c(\Lambda)>0$ such that 
  $\int_{-\infty}^b \exp(-\frac 1 2 \sigma_{\min} x^2)\,\mathrm{d}x
  \geq |b|^{-1}c(\Lambda)\exp(-\frac 1 2 \sigma_{\min} b^2)$ for all $b < -1$,
  we further get
  \begin{align*}
   & \left[\int_{-\infty}^{z^*_{\tau^c} - \hat \mu} 
                     \exp\left(- \frac 1 2 x_{\tau^c}^\top \hat R^{-1} x_{\tau^c}\right)  \, \mathrm d x_{\tau^c} \right]^{-2/b_{\min}^2} \\
  \leq{}& \left[\int_{(-\infty,b_{\min}) \times \ldots \times (-\infty,b_{\min})} 
                  \exp\left(-\frac 1 2 \sigma_{\max} \sum\nolimits_{i \in \tau^c} x_i^2\right) \, \mathrm{d} x_{\tau^c}\right]^{-2/b_{\min}^2}\\
  \leq{}& \left[ \frac{c(\Lambda)}{|b_{\min}|} \exp\left(- \frac 12 \sigma_{\max} b_{\min}^2\right)\right]^{-2|\tau^c|/b_{\min}^2}
  {}\leq{} \max\{1,1/c(\Lambda)\}^{2\tau^c} |b_{\min}|^{2\tau^c} \exp\left(\sigma_{\max} |\tau^c|\right).
  \end{align*}
  As $b_{\min}^2 \leq \tilde c(\Lambda) (1 + \sum_{i=1}^k |\log(z_i/z_1)|^2)$ for 
  an appropriate constant $\tilde c(\Lambda)$, there exists a constant $K(\Lambda,\tau)$
  depending continuously on $\Lambda$ such that
  \begin{align*}
   \left| \partial_{\lambda^2_{j,m}} \log \Phi_{|\tau^c|}(z^*_{\tau^c} - \hat \mu;\hat R)\right|
  \leq{}& K(\Lambda,\tau) \cdot \left(1 + \sum_{i=1}^k |\log(z_i/z_1)|^{2+2|\tau^{C}|}\right) \\
  \leq{} K(\Lambda,\tau) \cdot \frac 2 \varepsilon \sum_{i=1}^k \max\left\{\frac{z_i}{z_1}, \frac{z_1}{z_i}\right\}^\varepsilon
  \leq{}& K(\Lambda,\tau) \cdot \frac {2k} \varepsilon \sum_{i=1}^k \left(\frac{z_i}{\|z\|_\infty}\right)^{-\varepsilon}
  \end{align*}
  for all $\varepsilon>0$.
\end{proof}
  
\end{document}